\documentclass[leqno,11pt]{amsart}
\usepackage{amsmath}
\usepackage{amsthm}
\usepackage[utf8]{inputenc}
\usepackage[T1]{fontenc}

\usepackage{amssymb}
\usepackage{amsthm}
\usepackage{amsfonts}
\usepackage{amsmath}
\usepackage{anysize}
\usepackage{hyperref}
\usepackage{bbm}
\usepackage{xcolor}
\usepackage{verbatim}
\usepackage{epsfig}  		
\usepackage{epic,eepic}     
\usepackage{tikz}
\usepackage{tikz-cd}
\usetikzlibrary{matrix, calc, arrows}

\usepackage{geometry} 
\def\i{\sqrt{-1}}
\def\del{\partial}
\def\dbar{\bar\partial}
\def\ddbar{\del\dbar}

\newcommand{\RR}{\mathbb{R}}
\newcommand{\CC}{\mathbb{C}}

\def\del{\partial}

\newtheorem*{theorem*}{Theorem}

\newtheorem{theorem}{Theorem}[section]

\newtheorem*{notation*}{Notation}

\newtheorem{lemma}[theorem]{Lemma}
\newtheorem{corollary}[theorem]{Corollary}
\newtheorem{remark}[theorem]{Remark}
\newtheorem{example}[theorem]{Example}
\newtheorem{proposition}[theorem]{Proposition}
\newtheorem{question}{Question}

\newtheorem{definition}[theorem]{Definition}

\title[Destabilizing curves for dHYM and Z-critical equations]{The set of destabilizing curves for deformed Hermitian Yang-Mills and Z-critical equations on surfaces}

\author{Sohaib Khalid}
\author{Zakarias Sjöström Dyrefelt}

\address{Sohaib Khalid \\ Scuola Internazionale Superiore di Studi Avanzati (SISSA), Via Bonomea, 265, 34136 Trieste TS, Italy.}
\email{mkhalid@sissa.it}

\address{Zakarias Sjöström Dyrefelt \\ Institut for Matematik and Aarhus Institute of Advanced Studies, Aarhus University, Ny Munkegade 118, 8000, Aarhus C, Denmark. }
\email{dyrefelt@aias.au.dk}


\begin{document}

\maketitle

\begin{abstract}

We show that on any compact Kähler surface existence
of solutions to the Z-critical equation can be characterized using a finite number of effective conditions, where the number of conditions is bounded above by the Picard number of the surface.This leads to a first PDE analogue of the locally finite wall-chamber decomposition in Bridgeland stability.

As an application we characterize optimally destabilizing curves for Donaldson’s J-equation and the deformed Hermitian Yang-Mills equation, prove a non-existence result for optimally destabilizing test configurations for
uniform J-stability, and remark on improvements to convergence results for certain geometric
flows. \\
\\

\end{abstract}

\section{Introduction}

\noindent The question of characterizing optimally destabilizing test configurations is central in the study of K-stability and the circle of ideas surrounding the Yau-Tian-Donaldson conjecture in K\"ahler geometry. It is natural to expect that analogous concepts can be of interest also for other stability notions and other geometric PDE.

 In this work we focus on the study of obstructions and existence criteria for the Z-critical equation introduced by Dervan-McCarthy-Sektnan \cite{Dervan, DMS, McCarthy} as a general framework extending and encompassing several geometric PDE of interest in K\"ahler geometry, including the important special case of the deformed Hermitian Yang-Mills (dHYM) equation (see \cite{CollinsXieYau} for an introduction and \cite{CollinsShi} for a survey in the context of algebro-geometric stability and references therein). The latter is moreover closely related to the J-equation introduced by Donaldson \cite{DonaldsonJobservation} which is realized as a `small volume limit' in the dHYM setting, while at the same time being closely related to the constant scalar curvature equation and associated stability notions (e.g. K-stability). Existence criteria for the J-equation have been studied in a large body of work (see, for example, \cite{Chen2000, SongWeinkove, LejmiGabor, CollinsGabor}) culminating in \cite{GaoChen} and \cite{DatarPingali, Song} with a characterization of existence in terms of uniform J-stability, mirroring the Yau-Tian-Donaldson conjecture, and relating stability to analytic conditions tested on subvarieties, as conjectured by Lejmi-Sz\'ekelyhidi \cite{LejmiGabor}. Likewise, a number of similar results are known for the dHYM equation \cite{GaoChen, ChuLeeTakahashi}. A further and key aspect of the Z-critical equation is that it aims to highlight connections to Bridgeland stability \cite{Bridgeland}, which has generated a lot of interest due to its conjecturally far-reaching connections with birational geometry and mirror symmetry. Indeed, their approach closely mirrors Bayer's notion of a polynomial central charge \cite{Bayer}. 
 
In what follows, let $\#$ denote any of the J-equation, deformed Hermitian Yang Mills (dHYM) equation, or Z-critical equation (note that we single out the dHYM equation to emphasize certain aspects specific to this much studied special case).
 
 \begin{question} \label{Question intro}
Verify the existence of a solution to $\#$ via a finite number of effective{/numerical} conditions.
 \end{question}
 
 \noindent 
 This is a natural question related to the existence of \emph{subsolutions}, which is open in general. Such considerations have certainly also been of interest for other geometric PDE, notably for extremal K\"ahler metrics \cite{Calabi}, where effective conditions and optimal destabilizers have been investigated in a variety of different situations, such as toric geometry and its generalizations, or K-stability for Fano varieties (see e.g. \cite{Gabortoric, BlumLiuZhou, Hisamotooptimal, ZZhang, DelcroixJubert, Delcroix, DelcroixUKs, Dervanoptimal, Takahashioptimal}) and a number of open questions remain. The general picture is that even if one manages to characterize existence via algebro-geometric stability (e.g. K-stability) this typically involves testing an infinite number of conditions, which are in general still far from being {satisfactorily understood}.
 
 The present work deals with Question \ref{Question intro} in the case of K\"ahler surfaces, providing first results in the direction of the J-equation, dHYM equation and Z-critical equation, which have the advantage that they all can be studied using obstructions arising from subvarieties, as opposed to those coming from general test configurations. {A secondary intention is moreover for this paper to serve as a `base case' in an induction on dimension argument in forthcoming work, illustrating wall-chamber decompositions for Z-critical equations in higher dimension}. 
 
 The starting point of this work is a well-known proposition of Song-Weinkove \cite[Proposition 4.5]{SongWeinkove} dealing with the finite number of curves along which the J-flow blows up, where these curves are produced using Siu decomposition of suitable currents of analytic singularities. A key idea of this work is to replace this Siu decomposition approach with Zariski decomposition, which opens up to several improvements.
 More precisely, given a geometric PDE $\#$ (= J-equation, dHYM equation or Z-critical equation), we obtain a map $$\tau_\#:V \times H^{1,1}(X,\RR) \to H^{1,1}(X,\RR), (d_{\#},\alpha) \mapsto \tau_\#(d_{\#},\alpha),$$ where $V$ is (a subset of) a finite-dimensional space of `initial data' $d$ (taken at the topological level). 
 For example, if $\#$ is the deformed Hermitian Yang-Mills equation, then $V$ is usually taken to be the Kähler cone $\mathcal C_X$, wherein one picks a Kähler class $d_{\#}=\beta$ and a Kähler form $\theta \in \beta$ as the so-called \emph{B-field}. If $\#$ is the Z-critical equation, then $d_{\#}= V$ is the space of all stability data $\Omega$ (see \cite[Section 2.1]{DMS}). The class $\alpha \in H^{1,1}(X,\RR)$ is the topological constraint on our sought-after solution to $\#$. Under natural hypotheses it then follows from well-known existence criteria \cite{CollinsGabor, GaoChen, DatarPingali, ChuLeeTakahashi, Song} that $\#$ can be solved for initial data $d_{\#}$, in the class $\alpha$, precisely if $\tau_\#(d_{\#},\alpha)$ is Kähler class. 
 
 A main goal of this work is to improve on these existence criteria. We will give precise statements in each of the respective cases of the J-equation, dHYM equation and Z-critical equation. However they all fall under the structure of the following `meta theorem':

\begin{theorem} 
[Meta theorem] \label{Thm main meta} Let $X$ be a compact K\"ahler surface and fix a compact subset $K$ (of a suitable subset of) of $V\times H^{1,1}(X,\RR)$. Then there exists a set of finitely many curves $\mathcal V_K = \left\{E_1,\cdots, E_\ell\right\}$ on $X$ depending only on $K$ and $\#$ such that the equation $\#$ with underlying initial data and topological constraint $(d_{\#},\alpha) \in K$ is solvable if and only if the following finitely many conditions hold:
$$
\int_{E_i} \tau_{\#}(d_{\#},\alpha) > 0, \ \ \forall i = 1,2,\dots,\ell
$$
All the curves $E_i \in \mathcal{V}_K$ have negative self-intersection, and if moreover the image of $K$ under $\tau_\#$ is contained in the convex hull of $k$ pseudoeffective classes $\alpha_j \in H^{1,1}(X,\mathbb{R})$, then the cardinality $|\mathcal{V}_K|$ of the set of `test curves' is bounded above by $k\rho(X)$. If $X$ is moreover projective, then the cardinality of $\mathcal V_K$ is bounded above by $k\rho(X)-k$.
\end{theorem}

\begin{remark}
\begin{enumerate}
\item {In other words, we prove the existence of a \emph{finite} set of `destablizing curves' to the numerical criteria \cite{CollinsGabor, GaoChen, DatarPingali, ChuLeeTakahashi, Song}, such that the set of curves can moreover be taken uniform across compact sets of initial data. Note that this is a novelty already in the much studied special case of the dHYM equation.}
     \item {The hypotheses imposed are precisely the natural ones studied in the general literature for both the dHYM and Z-critical equations, and the `meta theorem' covers these cases of interest. More precisely, the above hypotheses hold automatically when specializing to the dHYM equation, and arise naturally from the point of view of `subsolutions' in the more general setting of the Z-critical equation (see \cite{DMS}).}
    \item{A main novelty in the strategy of proof is the replacement of Siu decomposition (e.g. as used by Song-Weinkove \cite{SongWeinkove}) by Zariski decomposition, which turns out to offer several advantages. In particular, it makes certain aspects of previous results for geometric flows more precise (in Section \ref{Section flow} we discuss the J-flow, line bundle mean curvature flow, and dHYM flow from this point of view).}
    \item \noindent In the case of the Z-critical equation Theorem \ref{Thm main meta} can be interpreted as a first PDE analogue of locally finite wall-chamber decompositions in Bridgeland stability. This is a main motivation behind the above result, as elaborated upon in Section \ref{Subsection intro Z-critical}.
     \item {The main technical simplifications that allow the strategy of proof to work on surfaces, is first the reduction to a complex Monge-Amp\'ere equation, and second, that intersection theory takes a particularly simple form in dimension $2$. It would however be natural to conjecture a similar statement in higher dimension $n \geq 3$, based on generalizations of Yau's solution of the Calabi conjecture \cite{Yau}, the characterisation of the Kähler cone due to Demailly-Paun \cite{DemaillyPaun}, and the Zariski decomposition of arbitrary pseudoeffective classes on compact complex manifolds due to Boucksom \cite{Boucksomthesis}. }
\end{enumerate}
\end{remark}

\noindent In each case when the main theorem \ref{Thm main meta} holds we moreover have the following corollary, which is new at least for the Z-critical equation (for general central charges in rank one): 

\begin{corollary} \label{Cor meta no neg curve intro}
If $X$ is a compact K\"ahler surface with no negative curves, then the equation $\#$ 
admits a solution for any admissible initial data $(d_{\#},\alpha)$ \emph{(see Corollary \ref{Cor neg curve dHYM intro} and Corollary \ref{Cor neg curve Z intro} for the precise statements)}.  
\end{corollary}

\noindent For the J-equation this is well-known (see \cite{DonaldsonJobservation}), and outside this result Corollary \ref{Cor meta no neg curve intro} provides classes of {new}
examples when the dHYM and Z-critical equations are solvable. {Significantly, in the case of the Z-critical equation, these examples occur away from the so-called `large volume limit'. This in turn has important implications for the locus of initial data where the Z-critical equation is solvable, presenting its boundary as a locally finite union of real codimension one `walls', in close analogy with related results in the theory of Bridgeland stability conditions (see e.g. \cite[Section 9]{BridgelandK3} for reference).}

\medskip

\subsection{Main results for the deformed Hermitian Yang-Mills equation} \label{Subsection intro dHYM} 
Let $X$ be a compact K\"ahler surface and pick two K\"ahler classes $(\alpha,\beta) \in \mathcal{C}_X \times \mathcal{C}_X$, where $\mathcal{C}_X$ denotes the open convex cone of K\"ahler cohomology classes in $H^{1,1}(X,\mathbb{R})$. The deformed Hermitian Yang Mills equation has a solution with respect to $(\beta,\alpha)$ if, for any Kähler metric $\theta \in \beta$, there exists a (1,1) form $\omega \in \alpha$ such that 
    \begin{equation} \label{Eq dHYM intro}
    \mathrm{Im}\left(e^{-\i \hat{\Theta}(\beta,\alpha)}(\theta + \i \omega)^2\right) = 0
    \end{equation}
    where $\hat\Theta(\beta,\alpha)$ is determined modulo $2\pi$ by the necessary topological contraint 
    $$
    \int_X\operatorname{Im}\left(e^{-\i\hat\Theta(\beta,\alpha)}(\beta + \i \alpha)^2\right)=0.
    $$
Our main result for the dHYM equation is the following: 
\begin{theorem} \label{Thm main dHYM intro}
Let $X$ be a compact K\"ahler surface and fix any compact subset $K \subseteq \mathcal{C}_X \times \mathcal{C}_X$.  Then there exists a finite set of curves of negative self-intersection $E_1, \dots, E_\ell$ on $X$, depending only on $K$, such that the following are equivalent: 
\begin{enumerate}
    \item The deformed Hermitian Yang Mills equation \eqref{Eq dHYM intro} admits a solution with respect to $(\beta,\alpha) \in K$.
    \item For all  irreducible curves $E \subseteq X$ we have
    $$ 
    \int_E \alpha + \cot(\hat\Theta(\beta,\alpha))\beta >0.
    $$
    \item For $i= 1, \cdots, \ell$ we have
    $$ 
    \int_{E_i} \alpha + \cot(\hat\Theta(\beta,\alpha))\beta >0.
    $$ 
\end{enumerate}
\end{theorem}

\begin{remark}
The new result here is $(2) \Leftrightarrow (3)$ stating that it is enough to test the numerical condition $(2)$ on a finite set of `test curves', where moreover this set of curves is completely uniform of the choice of $(\alpha,\beta) \in K$ across any given compact subset $K$.
This is a consequence of our proofs using Zariski decomposition, rather than Siu decomposition of currents. In fact, in the case of the dHYM equation the set of curves can moreover be taken independent of the scaling $\alpha \mapsto k\alpha$ as $k \rightarrow +\infty$, so that the same `test curves' can be used for all rescalings of the dHYM equation including its small volume limit, the J-equation, as explained in Section \ref{Subsection optimal dest Jstab}. The equivalence $(1) \Leftrightarrow (2)$ is due to \cite{JacobYau}.
\end{remark}

\noindent As a consequence of $(2) \Leftrightarrow (3)$ one may note the following existence result, analogous to a well-known result of Donaldson \cite{DonaldsonJobservation} for the J-equation:

\begin{corollary} \label{Cor neg curve dHYM intro}
Suppose that $X$ is a compact K\"ahler surface with no curves of negative self-intersection. Then the deformed Hermitian Yang Mills equation can always be solved for all pairs $(\beta,\alpha)$ of Kähler classes and every Kähler form $\theta \in \beta$. 
\end{corollary}

\noindent It is moreover interesting to note the following implications for the structure of the `dHYM-positivity locus': 

\begin{corollary} \label{Cor dHYM structure intro}
In the setup of Theorem \ref{Thm main dHYM intro}, let 
$$
\widetilde{c_{\beta,\alpha}} := \frac{\int_X(\alpha^2 - \beta^2)}{2\int_X\alpha\cdot\beta} \in \mathbb{R}
$$
and
$$
\tau_\mathrm{dHYM}(\beta,\alpha) := \alpha - \widetilde{c_{\beta,\alpha}}\beta \in H^{1,1}(X,\mathbb{R}). 
$$
Consider moreover the set 
$$
\mathcal{U} := \left\{(\beta,\alpha) \in \mathcal{C}_X \times \mathcal{C}_X: \textrm{\emph{The dHYM equation \eqref{Eq dHYM intro} is solvable with respect to  }} (\beta,\alpha). \right\}
$$
 of $(\beta,\alpha) \in \mathcal{C}_X \times \mathcal{C}_X$ such that equation \eqref{Eq dHYM intro} admits a solution $\omega \in \alpha$. Fixing any compact set $K \subseteq \mathcal{C}_X \times \mathcal{C}_X$, the restriction $\mathcal{U}_{\vert K}$ is cut out by a finite number of open conditions depending only on $K$, namely
$$
\mathcal{U}_{\vert K} = \left\{(\beta,\alpha) \in K:  \int_{E_i}  
\tau_\mathrm{dHYM}(\beta,\alpha) > 0, i = 1, \cdots, \ell\right\},
$$
where the curves $E_i$ are those appearing in Theorem \ref{Thm main dHYM intro}.
In particular, the boundary $\partial\mathcal{U}_{\vert K}$ is defined by
$$
\partial\mathcal{U}_{\vert K} = \left\{(\beta,\alpha) \in K: \prod_{i = 1}^\ell \int_{E_i}   \tau_\mathrm{dHYM}(\beta,\alpha) = 0 \right\},
$$ 
which forms a real algebraic set in $K$ {of real codimension one}, {giving rise to a locally finite wall-chamber decomposition of the space of dHYM equations.}
\end{corollary}

\begin{remark}
Note that this immediately implies that dHYM-positivity (and thus solvability of the dHYM-equation) is an open condition under perturbation of the underlying data $(\beta,\alpha) \in \mathcal{C}_X \times \mathcal{C}_X$. Such openness would already follow from general theory, since the dHYM equation reduces to a complex Monge-Amp\`ere equation. Nonetheless, the above Corollary \ref{Cor dHYM structure intro} makes this openness rather explicit. 
\end{remark}

\noindent For the J-equation we have an analogous statement, see Theorem \ref{Thm main Jequation}.

\subsection{Main results for the Z-critical equation and a wall-chamber
decomposition in the spirit of Bridgeland} \label{Subsection intro Z-critical}

In \cite[Section 2.1]{DMS} Dervan-McCarthy-Sektnan introduced the notion of a \emph{Z-critical connection} on a holomorphic vector bundle $E \to X$ on a compact Kähler manifold $X$ of dimension $n$, where Z is a choice of a \emph{polynomial central charge}, as a differential geometric analogue (in the case of vector bundles) of the abstract notion of a Bridgeland stability condition (introduced in \cite{Bridgeland}). In their setup, the choice of a polynomial central charge $Z = Z_\Omega$ is the data of a triple $\Omega = (\beta, \rho,U)$ where $\beta$ is a Kähler class, $\rho = (\rho_0,\rho_1, \cdots, \rho_n) \in \mathbb C^{n+1}$ is a \emph{stability vector} such that $\rho_d/\rho_{d+1}$ and $\rho_n$ lie in the upper half-plane, and $U \in \bigoplus_i H^{i,i}(X,\RR)$ is \emph{unipotent cohomology class}, i.e. a mixed-degree cohomology class whose degree zero part is equal to 1. With this notation, for a fixed Kähler metric $\theta \in \beta$ and a fixed smooth representative $\tilde U \in U$, set $ \tilde Z_\Omega(E,h) $ to be the degree $(n,n)$ part of the complex endomorphism-valued form $$ \left(\sum_{d = 0}^n \rho_i \theta^d\right)\wedge \tilde U \wedge \operatorname{ch}(E,h) $$ where $h$ is a Hermitian metric on $E$ and $\operatorname{ch}(E,h)$ is the Chern-Weil representative of the total Chern character $\operatorname{ch}(E)$ of $E$ with respect to $h$. Then, the \emph{Z-critical equation} is written \begin{equation}
    \label{Eqn Z-Critical}
    \operatorname{Im}\left(e^{-\i\varphi(E)}\tilde Z_\Omega(E,h)\right) = 0
\end{equation} 
where the metric $h$ is understood as the desired solution. Here, $e^{-\i\varphi(E)}$ is the unique cohomological constant that, modulo $2\pi$, satisfies $$\varphi(E) = \arg\left(\int_X \left(\tilde Z_\Omega(E,h) \right)\right)$$ and the imaginary part is understood as the anti-self adjoint component of the form with respect to the metric $h$. More precisely, there is a map $(\cdot)^\dagger:\mathcal A^{n,n}(\operatorname{End}E) \to \mathcal A^{n,n}(\operatorname{End}E)$ that sends any (complex) endomorphism-valued $(n,n)$-form to its Hermitian conjugate. Then the imaginary part of a form $\Theta$ is simply $$\operatorname{Im}(\Theta) = \frac{1}{2\i}(\Theta - \Theta^\dagger). $$ \\
As mentioned above, the Z-critical equation is meant to capture, in the context of differential geometry, the abstract notion of a Bridgeland stability condition on the derived category of coherent sheaves of a smooth complex projective variety. From this perspective, one should expect to see features of the theory of Bridgeland stability arise purely in terms of the differential geometry of the Z-critical equation. One such feature is a locally finite wall-chamber structure of the manifold parametrizing stability data (see, for example \cite[Section 9]{BridgelandK3}). The following result is well-motivated from this point of view, and can be seen as a first analogue of a locally finite wall-chamber structure on the side of differential geometry. (In particular, see \cite[Proposition 9.4]{BridgelandK3}, which gives a closely analogous result in the case of K3 surfaces, and compare with Corollary \ref{Cor Z-critical structure} below.)

{To state the result we recall the \emph{volume form hypothesis} introduced by the Dervan-McCarthy-Sektnan \cite{DMS} {as a natural analogue of the supercritical condition for the dHYM equation, and the natural setting for studying the Z-critical equation and its associated stability}. Let $X$ be a compact Kähler surface and $L\to X$ a holomorphic line bundle on $X$. Fix a choice $Z= Z_\Omega$ of polynomial central charge corresponding to the stability datum $(\beta,\rho,U)$ as above, and choose representatives $\theta \in \beta$ and $\tilde U \in U$. Then the left-hand-side of the Z-critical equation \eqref{Eqn Z-Critical} can be re-written as 
$$
\operatorname{Im}\left(e^{-\i\varphi(L)}\tilde Z_\Omega(L,h)\right) = c(\chi_h^2 + \chi_h\wedge\tilde\eta + \tilde\gamma) 
$$
where the forms $\tilde\eta$ and $\tilde\gamma$ and the constant $c$ depend only on the lifts $\tilde U$ and $\theta$ of the stability data and the line bundle $L$, and $\chi_h$ is the Chern-Weil representative of $c_1(L)$ with respect to the unknown metric $h$. Then the lift $(\beta,\rho,\tilde U)$ is said to satisfy the \emph{volume form hypothesis} at $L$ if $c\neq 0$ and the (2,2)-form 
$$
\frac{1}{4}\tilde\eta^2 - \tilde\gamma \in \mathcal A^{2,2}(X)
$$
is a volume form on $X$. In fact, one can always find such a choice of lifts provided that the cohomological conditions $c \neq 0$ and  
$$
\operatorname{Vol}(\Omega,L):= \int_X \left(\frac{1}{4}\tilde\eta^2 - \tilde\gamma\right) > 0
$$
are satisfied (Corollary \ref{Cor positive volume implies volume form}). In this case, we say that $\Omega$ \emph{has positive volume at $L$}. Let $\eta(\Omega,L),\gamma(\Omega,L)$ denote the cohomology classes of $\tilde \eta = \tilde\eta(\Omega,L)$ and $\tilde\gamma = \tilde\gamma(\Omega,L)$ respectively. 

{The following main result is stated under the full generality of the positive volume condition:} 
}

\begin{theorem}
\label{Thm Main Theorem Z Critical Equation}
Let $X$ be a compact Kähler surface, $K\subseteq \mathcal C_X \times (\CC^*)^{3}\times \bigoplus_k H^{k,k}(X,\RR)$ be a compact set of polynomial central charges on $X$, and $L\to X$ be a holomorphic line bundle. Suppose that each $\Omega = (\beta,\rho,U)  \in K$ has positive volume at $L$. Then, there exists a finite set of curves of negative self-intersection $E_1,\cdots,E_\ell$ such that, for each $\Omega = (\beta,\rho,U) \in K$ the following conditions are equivalent: \begin{enumerate}
    \item There exists a choice of lifts of $\Omega$ such that the line bundle $L \to X$ admits a solution to the Z-critical equation \eqref{Eqn Z-Critical}.
    \item The class $$ \tau(\Omega,L)  := s\left(c_1(L) + \frac{1}{2}\eta(\Omega,L)\right)$$ satisfies $$ \int_{E_i} \tau(\Omega,L) > 0 \quad \forall i = 1,\cdots, \ell$$ where $s$ is the sign of the non-zero real number $$\int_X \left(c_1(L)+\frac{1}{2}\eta(\Omega,L)\right)\cdot \beta \in \RR.$$
\end{enumerate}
\end{theorem}

\noindent As a consequence we derive the following analogue of a certain wall-chamber decomposition in Bridgeland stability {(since a first version of this manuscript, chamber decompositions have also been studied for canonical metrics in complex geometry, using different approaches related to moduli of cscK manifolds and K-moduli, see e.g. \cite{SektnanTipler, CZhou2, CZhou})}:

\begin{theorem}
    Let $X$ be a compact Kähler surface and $S$ any finite set of holomorphic line bundles on $X$. Let $V_S$ denote the set of stability data $\Omega$ that have positive volume at each $L \in S$. Then, $V_S$ admits a locally finite wall-chamber decomposition.
\end{theorem}

\noindent We refer the reader to Corollary \ref{Cor Z-critical structure} for precise definitions and details.

\begin{remark}
The analogous statement holds more generally when $L$ is replaced by an arbitrary $(1,1)$-cohomology class $\alpha \in H^{1,1}(X,\RR)$. In that case, one replaces $\mathrm{ch}(L,h)$ in equation \eqref{Eqn Z-Critical} with $\exp(\chi):= 1 + \chi + \frac{1}{2}\chi^2$, where $\chi \in \alpha$ is the sought-after (1,1) form which solves the equation.
\end{remark}

\noindent As before, we have the following immediate corollary, which in the case of the Z-critical equation is new to the literature and gives examples of Z-critical metrics away from the so-called `large volume limit':  

\begin{corollary} \label{Cor neg curve Z intro}
    Let $X$ and $L$ be as in the Theorem \ref{Thm Main Z Critical Proof}. Suppose $X$ does not admit any curves of negative self-intersection. Let $\Omega$ be a choice of stability data defining a polynomial central charge $Z_\Omega$ and such that $V(\Omega,L) > 0$. Then, for any lift $\tilde\Omega$ satisfying the volume form hypothesis, the $Z_\Omega$-critical equation admits a solution on $L$.
 \end{corollary}
 
\noindent An analogous result to Corollary \ref{Cor dHYM structure intro} also holds for the Z-critical equation, see Corollary \ref{Cor Z-critical structure}. 

Finally, let us remark that we hope to apply similar techniques in higher dimension in the future (especially in dimension $3$, where Pingali \cite{Pingali} has remarked on the close relationship between the dHYM equation and reductions to Monge-Amp\'ere type equations). Such generalizations involve several additional hurdles related to the big cone and numerical properties of the Zariski decomposition and requires the development of further new techniques.

\smallskip

\subsection{Optimal destabilizers for the J-equation and dHYM equation}

Motivated by well-known (often conjectural) connections between existence of solutions to geometric PDEs and suitable algebraic or analytic stability notions, we also ask about \emph{optimal destabilizers} to the respective equations above. 

On the analytic side one asks for curves that optimally destabilize the relevant positivity condition in the following sense: an irreducible curve $E_0 \subseteq X$ is said to be \emph{optimally dHYM-destabilizing} if
$$
    \int_{E_0} \tau_{\mathrm{dHYM}}([\theta],[\omega]) = 
    \inf_{E \subseteq X} \left( \int_E 
    \tau_{\mathrm{dHYM}}([\theta],[\omega]) \right) \leq 0
$$
i.e. $E_0$ is a destabilizer, which is moreover `most destabilizing' among the set of all curves (and as before $\tau_{\mathrm{dHYM}}([\theta],[\omega]) = [\omega] + \cot(\hat\Theta([\theta],[\omega])$). A similar definition can be given of \emph{optimally J-destabilizing} curves using $\tau_{\mathrm{J}}$ instead (see Section \ref{Section J}). From the main results (Theorem \ref{Thm main dHYM intro} and \ref{Thm main Jequation}) the set of optimally destabilizing curves are finite, in each of these cases, and the following result shows that they are in fact closely related.

\begin{theorem} \label{Thm opt J vs opt dHYM intro}
Suppose that 
$([\theta],[\omega]) \in H^{1,1}(X,\mathbb{R})$ such that $[\theta]^2 < [\omega]^2$.
Then the infima 
$$
\inf_{E \subseteq X} \left( \int_E
    \tau_{\mathrm{dHYM}}([\theta],[\omega]) \right)
$$
and
$$
\inf_{E \subseteq X} \left( \int_E
    \tau_{\mathrm{J}}([\theta],[\omega]) \right)
$$
are realized by the same finite set of curves
\begin{equation} \label{Eq opt dest curves}
\left\{E \subseteq X \;: \int_E \bf{a}([\theta],[\omega]) = 0 \right\},
\end{equation}
where $\mathbf{a}([\theta],[\omega]) \in H^{1,1}(X,\mathbb{R})$ is the $(1,1)$-class defined in Proposition \ref{Prop optimally Jneg curves}.
In particular, if the sets of optimally J/dHYM-destabilizing curves are non-empty, then they coincide, and are explicitly given by \eqref{Eq opt dest curves}.  
\end{theorem}

\noindent In other words, the numerical existence criteria for J and dHYM equations can be tested for the \emph{same} finite set of curves. As an alternative point of view, this shows that the set of optimally dHYM-destabilizing curves can be chosen uniformly across any rescaling $\alpha \mapsto k\alpha$, $k > 0$, to the small volume limit, and these are precisely the curves that witness the `non-positivity' of the class $\mathbf{a}([\theta],[\omega])$ in the Nakai-Moishezon criterion. {We stress that while the sets of \emph{optimally} J/dHYM-destabilizing curves coincide when non-empty, this does of course \emph{not} imply the existence of any destabilizing curves, nor does it imply that the dHYM and J-equations are solvable at the same time (as in the existence criteria one integrates \emph{different} $(1,1)$-forms $\tau_{\mathrm{J}}$ and $\tau_{\mathrm{dHYM}}$ over these same curves).} 

In the case of optimal destabilizers on the side of algebro-geometric stability, we show that the situation is quite different to that of curves. More precisely, {building crucially on work of Hattori \cite{Hattori}} we observe that optimally destabilizing \emph{test configurations} {(defined with respect to the minimum norm, see \eqref{Eq opt dest tc})} never exist across rational classes:  

\begin{theorem} \label{Thm no opt dest intro}
Let $X$ be a smooth projective surface and $L,H$ $\mathbb{Q}$-line bundles on $X$. Then no optimally destabilizing normal and relatively K\"ahler test configuration exists for $(X,L,H)$.
\end{theorem}

\noindent This is a straightforward extension of \cite[Theorem 7.3]{Hattori} together with \cite[Theorem 1]{SD4}, in that we take \emph{optimally destabilizing} to mean that the infimum $$\inf_{\Phi \in \mathcal{H}^{\mathrm{NA}}(X,[\omega])}\frac{\mathrm{E}_{\theta}^{\mathrm{NA}}(\Phi)}{||\Phi||}$$ is realized by a test configuration $\Phi_0$ (even when this infimum is non-zero). Note that while this definition is not identical to that used in e.g. \cite{Gabortoric}, based on Donaldson's lower bound of the Calabi functional \cite{DonaldsonCalabi} and involving the $L^2$ norm, the two notions are very closely related. As a further result we also observe the following: 

\begin{theorem} \label{Thm J vs UJ intro}
Let $X$ be a smooth projective surface. Then the following are equivalent: 
\begin{enumerate} 
\item X admits at least one negative curve. \item There exist $L$ and $H$ ample line bundles on $X$ such that $(X,L)$ is $J^H$-stable but not uniformly $J^H$-stable.
\end{enumerate}
\end{theorem}

\begin{remark}
The above sheds additional light on a recent example in \cite{Hattori}, where they provided a polarized surface $(X,L)$ which is J-stable but not uniformly. In fact, by Theorem \ref{Thm J vs UJ intro} such examples exist in abundance (namely, precisely on any projective surface $X$ carrying a curve of negative self-intersection). As shown in \cite[Corollary 7.5]{Hattori} such examples can be useful in constructing examples of K-stable but not uniformly K-stable polarized normal pairs, touching upon fundamental questions relating to the Yau-Tian-Donaldson conjecture for constant scalar curvature metrics. {Note however that Hattori's construction applies to very singular surfaces so that the entropy term in the non-Archimedean Mabuchi functional vanishes. Producing examples of K-stable but not uniformly K-stable smooth surfaces appears to be very challenging and would likely require new ideas.}
\end{remark}

\noindent The conclusion is that on the analytic side optimally destabilizing curves exist (in a finite number, at most $\rho(X)$) and one can find a finite set of curves that works uniformly across all underlying data in a given compact subset of the K\"ahler cone. On the other hand, on the algebraic side optimally destabilizing test configurations for J-stability can never exist in the sense of \eqref{Eq opt dest tc}
across rational classes. (We are currently unable to determine whether this is also true for $\mathbb{R}$-line bundles).

\subsection{Flow aspects}

As a final application of the main Theorem \ref{Thm main meta} we remark on improvements to certain results on singularities of the J-flow, the line bundle mean curvature flow, and a dHYM flow, due to Song-Weinkove \cite{SongWeinkove}, Takahashi \cite{TakahashiLMBCF} and Fu-Yau-Zhang \cite{FuYauZhang} respectively. We refer the reader to Section \ref{Section flow} for the statements and proofs.

\subsection{Outline of the paper} 

Section \ref{Section prelim} deals with preliminaries. 

Section \ref{Section main thms} includes proofs of the main theorems for the dHYM equation (Theorem \ref{Thm Main dHYM Proof}) and the Z-critical equation (Theorem \ref{Thm Main Z Critical Proof}). Note that the former is a consequence of the latter, 
but we write out both proofs for clarity of exposition. We also state Corollary \ref{Cor Z-critical structure}.

In Section \ref{Section J} we first state and prove the main theorem (Theorem \ref{Thm main Jequation}) in the special case of the J-equation, characterize the set of optimally destabilizing curves for the J-equation and the dHYM equation (Theorem \ref{Thm opt J vs opt dHYM intro}), and prove non-existence (in general) of optimally destabilizing test configurations for the J-equation, which also involves comparing J-stability and uniform J-stabiity (Theorem \ref{Thm J vs UJ intro}).

In Section \ref{Section example} we exemplify aspects of Corollary \ref{Cor Z-critical structure} for $\mathbb{P}^2$ blown up in two points.

In Section \ref{Section flow} we finally apply the main theorem \ref{Thm main meta} to state and discuss improvements to certain results on singularity formation and convergence of geometric flows. 

\medskip

\subsection{Acknowledgements}

\noindent We are grateful to Jacopo Stoppa, Ruadha\'i Dervan, Lars Martin Sektnan, Sébastien Boucksom, and the anonymous referees for their feedback and many helpful comments and suggestions. 
The second named author is supported by the European Union’s Horizon 2020 research and innovation programme under the Marie Skłodowska-Curie grant agreement No 754513 and Aarhus University Research Foundation.

\section{Preliminaries} \label{Section prelim}

\noindent In this section we recall certain well-known and classical facts about the theory of positivity on Kähler surfaces. We will nevertheless state results in more generality than we need them, so that we can emphasise that the methods can be brought to bear upon problems in higher dimensions, though with numerous additional difficulties that would be interesting to address in future work. 

\subsection{Positive cones in cohomology}
Let $X$ be a compact Kähler manifold of dimension $n$. Recall that the \emph{Kähler cone} $\mathcal C_X$ is the convex open cone in $H^{1,1}(X,\RR)$ comprising precisely those cohomology classes, called \emph{Kähler classes}, which can be represented by a Kähler metric. The closure $\mathcal N_X := \overline{\mathcal C_X}$ of $\mathcal C_X$ is called the \emph{nef cone}, comprising the \emph{nef classes}. \\

Recall that the canonical map from $H^{1,1}(X,\RR)$ to the space of real-valued $d$-closed (1,1) currents modulo $d$-exact (1,1) currents (induced by taking a (1,1) form to the current it defines) is an isomorphism. Thus, it makes sense to talk about classes in $H^{1,1}(X,\RR)$ being represented by currents. With this in mind, a class $\tau \in H^{1,1}(X,\RR)$ is called \emph{pseudoeffective} if it can be represented by a closed positive current, and is called \emph{big} if it can be represented by a \emph{Kähler current}, that is a closed positive current $T$ such that 
$
T \geq \varepsilon \omega
$ 
for some $\varepsilon > 0$ and some choice of Kähler form $\omega$. Note that if $V = \sum_i a_i V_i$ is any effective $\RR$-cycle of closed analytic subvarieties $V_i \subseteq X$ of codimension one, then the class $[V] \in H^{1,1}(X,\RR)$ of $V$ is pseudoeffective, since it admits the closed positive (1,1) current of integration 
$$
\mathcal A^{n-1,n-1}(X,\RR) \ni \theta \mapsto \sum_i a_i \int_{V_i} \theta.
$$  
The pseudoeffective and big classes respectively comprise the \emph{pseudoeffective cone} $\mathcal E_X$ and the \emph{big cone} $\mathcal B_X$. Clearly, we have $\overline{\mathcal B_X}=\mathcal E_X$. It is also easy to see from the definitions that 
$$
\mathcal C_X \subseteq \mathcal B_X\subseteq \mathcal E_X, \;  \textrm{ and } \; \overline{\mathcal C_X} = \mathcal N_X \subseteq \mathcal E_X.
$$ The cone $\mathcal E_X$ induces on $H^{1,1}(X,\RR)$ a partial ordering: we will write $\alpha \leq \beta$ precisely when $\beta-\alpha \in \mathcal E_X$. (Note that this differs slightly from the usual convention of using $\mathcal C_X$.)\\ 

We also recall the following notion, due to Boucksom (see \cite[Definition 2.2]{Boucksomthesis}), of classes that are `nef in codimension one'. A class $\alpha \in H^{1,1}(X,\RR)$ is called a \emph{modified nef class} if, for every $\varepsilon > 0$ and some (hence any) choice of Kähler metric $\omega$ on $X$, there exists a closed (1,1) current $T$ representing $\alpha$ such that $T + \varepsilon \omega \geq 0$ and the Lelong number $\nu(T,D) = 0$ along any analytic subvariety $D \subseteq 
 X$ of codimension one. The modified nef classes comprise $\mathcal{MN}_X$, the \emph{modified nef cone}.\\

The \emph{Neron-Severi group} $NS(X)_\RR$ of $X$ is the subpace of $H^{1,1}(X,\RR)$ generated by the first Chern classes of holomorphic line bundles on $X$. Throughout this paper, we shall denote its dimension by $\rho(X) : = \dim_\RR NS(X)_\RR.$ Clearly, $\rho(X) \leq h^{1,1}(X)$.\\

Now suppose $X$ is a compact Kähler surface. The intersection pairing on $H^{1,1}(X,\RR)$: 
$$
(\alpha,\beta)\mapsto \int_X \alpha\cdot\beta
$$
is non-degenerate and, according to the Hodge Index Theorem, has precisely one positive eigenvalue, and so is a Lorentzian pairing defined on $H^{1,1}(X,\RR)$. The \emph{positive cone} $\mathcal P^+_X$ with respect to this pairing is the set of classes $\tau \in H^{1,1}(X,\RR)$ such that 
$$
\int_X \tau^2 > 0, \; \textrm{ and } \; \int_X \tau\cdot\alpha > 0
$$
for some (hence any) fixed Kähler class $\alpha \in \mathcal C_X$. A straightforward calculation shows that $\mathcal P^+_X$ is self-dual. Moreover, from the definition of $\mathcal C_X$ and $\mathcal B_X$, it follows easily that $\mathcal C_X \subseteq \mathcal B_X^*$ and so $\mathcal N_X \subseteq \mathcal E_X^*$. In fact, the following generalisation of the classical Nakai-Moishezon criterion due to Lamari implies that $\mathcal N_X = \mathcal E_X^*$. (See \cite[Theorem 5.3]{Lamari}.)
\begin{theorem}[Lamari] Suppose $X$ is a compact Kähler surface and $\beta$ is any non-zero class in the closure of $\mathcal P^+_X$. Then, a class $\alpha \in H^{1,1}(X,\RR)$ is a Kähler class if and only if it satisfies the following. 
\begin{enumerate}
    \item $\int_X \alpha^2 > 0$,
    \item $\int_X \alpha\cdot\beta > 0$,
    \item For every irreducible curve $E \subseteq X$ of negative self-intersection, we have 
    $$
    \int_E \alpha > 0.
    $$
\end{enumerate}
\end{theorem}
Moreover, the duality $\mathcal N_X = \mathcal E_X^*$ implies the following familiar fact, which we use repeatedly. 
\begin{lemma}\label{Lemma positive inside big}
    On a compact Kähler surface $X$, we have $\mathcal{P}_X^+ \subseteq \mathcal{B}_X$.
\end{lemma}
    \begin{proof}
    We have $\mathcal N_X \subseteq \overline{\mathcal P^+_X}$. By duality, we conclude $\overline{\mathcal P^+_X} = \overline{\mathcal P^+_X}^* \subseteq \mathcal N_X^* = \mathcal E_X$, since $\overline{\mathcal P^+_X}$ is self-dual. This shows that $\overline{\mathcal P^+_X} \subseteq \overline{\mathcal B_X}$, and the result about the interior $\mathcal P^+_X$ of $\overline {\mathcal P^+_X}$ follows.
    \end{proof}

\subsection{The Zariski decomposition}
A classical theorem of Zariski \cite{Zariski} states that any effective divisor $D$ on a surface can be uniquely decomposed as a sum of $\mathbb Q$-divisors $D = Z + N$ where Z is nef, $N$ has negative-definite intersection matrix and $Z \cdot N = 0$. In \cite{Boucksomthesis} Boucksom generalised this decomposition to arbitrary pseudoeffective classes on compact complex manifolds. His result, in the context of Kähler manifolds, is the following. (See \cite[Sections 3.2, 3.3]{Boucksomthesis}.)
\begin{theorem}[Boucksom, \cite{Boucksomthesis}]\label{Thm Zariski decomposition} Let $X$ be a compact Kähler manifold of dimension $n$. There exist uniquely determined homogeneous maps $Z:\mathcal E_X \to \mathcal E_X$, and $N: \mathcal E_X \to \mathcal E_X$ satisfying the following properties. 
\begin{enumerate}
    \item For every $\tau \in \mathcal E_X$, $Z(\tau) \in \mathcal{MN}_X$ is a modified nef class.
    \item For every $\tau \in \mathcal E_X$, $N(\tau)$ is the class of a unique effective $\RR$-divisor, i.e. 
    $$
    N(\tau) = \sum_{i = 1}^\ell a_i [E_i]
    $$
    for some $a_i > 0$ and $[E_i]$ the classes of analytic subvarieties $E_i \subseteq X$ of codimension one forming an \emph{exceptional family}. The current of integration along the divisor $\sum a_i E_i$ is the unique closed positive current in the class $N(\alpha)$. Moreover, the classes $[E_i]$ are linearly indepenent in $H^{1,1}(X,\RR)$. In particular, $\ell \leq \rho(X)$.
    \item For every $\tau \in \mathcal E_X$, we have 
    $$
    \tau = Z(\tau) + N(\tau).
    $$
    \item The map $N$ is convex. 
\end{enumerate}
    
\end{theorem}

\noindent In our case of interest, where $X$ is a compact Kähler surface, the nef cone $\mathcal N_X$ and the modified nef cone $\mathcal{MN}_X$ coincide. (See, for example, \cite[Theorem 4.1]{Boucksomthesis}.) Moreover, the prime divisors $E_i$ (i.e. curves) appearing in Theorem \ref{Thm Zariski decomposition} (2) have a negative-definite intersection matrix (and this property characterises the set of \emph{exceptional families} referred to in the above Theorem \ref{Thm Zariski decomposition}), see \cite[Theorem 4.5]{Boucksomthesis}). With the following example we recall that the Zariski decomposition of a given class can be computed explicitly without knowing the entire nef cone:

{\begin{example}
    Let $X$ be the Fermat quartic hypersurface in $\mathbb P^3$ given by the locus $x_0^4 + x_1^4 + x_2^4 + x_3^4 = 0$. Let $\xi = \exp(2\pi \i/8)$, which is a primitive eighth root of unity. Then, writing 
    $$
    x_0^4 + x_1^4 + x_2^4 + x_3^4 = \left(\prod_{j\in\left\{1,3,5,7\right\}}(x_0 - \xi^j x_1)\right)+ \left(\prod_{j\in\left\{1,3,5,7\right\}}(x_2 - \xi^j x_3)\right)
    $$
    we see immediately that the sixteen lines 
    $$
    \ell(j_1,j_2):= \left\{ [\xi^{j_1}x_1:x_1: \xi^{j_2}x_3:x_3] \in \mathbb P^3 \; | \;[x_1:x_3] \in \mathbb P^1\right\} \quad j_1,j_2 \in \left\{1,3,5,7\right\}
    $$
    lie on $X$. Since $X$ is a smooth degree $4$ hypersurface, by adjunction we get that $X$ is a K3 surface, so the genus formula for an irreducible curve $C \subseteq X$ reduces to 
    $$
    2g(C) - 2 = C^2
    $$
    from which it follows that the only curves of negative self-intersection on $X$ are smooth rational curves, all of which have self-intersection equal to $-2$. Now note that $\ell(j,k)$ and $\ell(r,s)$ meet transversely in a point precisely when $j = r$ or $k=s$ but not both. So, for fixed and distinct $j,k \in \left\{1,3,5,7\right\}$, any three of the four lines $\ell(j,j), \ell(j,k),\ell(k,k),\ell(k,j)$ have an intersection matrix (with respect to an appropriate basis) given by 
    $$
    \left(\begin{matrix}
        -2 & 1 & 0 \\
        1 & -2 & 1 \\
        0 & 1 & -2
    \end{matrix}\right),
    $$
    which is negative-definite, and hence form an exceptional family. On the other hand, it is easy to verify that the divisor 
    $$
    P(j,k):= \ell(j,j)+\ell(j,k)+\ell(k,k)+\ell(k,j)
    $$
    is nef. So, it follows that any convex combination 
    $$
    \tau = a[\ell(j,j)]+b[\ell(j,k)]+c[\ell(k,k)]+d[\ell(k,j)] \in H^{1,1}(X,\RR)
    $$
    where $a, b, c, d \geq 0$ are non-negative real numbers has its Zariski decomposition given by
    $$
    Z(\tau) = \min(a,b,c,d)[P(j,k)]
    $$
    and 
    $$
    N(\tau) = \alpha-\min(a,b,c,d)[P(j,k)],
    $$
    since the latter, by the above, has irreducible components that form an exceptional family.
\end{example}

\subsection{Zariski decomposition and birational transformations} 
The Zariski decomposition is moreover well-behaved under birational transformations. More precisely, we have the following. 
\begin{lemma}\label{Lem Zariski Birational}
Suppose $X$ is a compact Kähler manifold and $\pi : Y \to X$ a birational morphism. Then, for any pseudoeffective class $\tau$ on $X$ we have 
$$
N(\tau) = \pi_*N(\pi^*\tau).
$$
In particular, if $\pi: Y \rightarrow X$ is the blowup of a compact Kähler surface $X$ at a point $p \in X$, with $E \subseteq Y$ the exceptional divisor, then for any pseudoeffecitve class $\tau$ on $X$, the support of $N(\pi^*\tau)$ is composed of the strict transform of the curves making up the support of $N(\tau)$, plus possibly the exceptional divisor $E$.
\end{lemma}
\begin{proof}
    The first part is \cite[Lemma 5.8]{Boucksomthesis}, and the second part follows immediately from the formula for the pullback of a divisor under a blowup map. 
\end{proof}

\subsection{Variational setup and stability} We will use the variational setup for the study of the J-equation introduced by \cite{CollinsGabor}, inspired by the analogous theories for K-stability and Calabi's canonical metrics (see \cite{Calabi, Mabuchi, Futaki, Tian, Donaldsontoric} and a large volume of further work). Regarding the algebro-geometric notions of J-stability we consider the definition of \cite{LejmiGabor}, but allow also the natural generalization of these notions to arbitrary compact K\"ahler manifolds, using the notion of test configuration in the sense of \cite{SD1,DervanRoss}. Following \cite{BHJ1} we denote the space of all relatively K\"ahler test configurations for $(X,[\omega])$ by $\mathcal{H}^{\mathrm{NA}}(X,[\omega])$, and often denote their elements by $\Phi$. The associated numerical invariant characterizing (uniform) J-(semi)stability with respect to given K\"ahler classes $([\theta],[\omega]) \in \mathcal{C}_X \times \mathcal{C}_X$ is denoted by $\mathrm{E}_{\theta, \omega}^{\mathrm{NA}}(\Phi)$, for $\Phi \in \mathcal{H}^{\mathrm{NA}}(X,[\omega])$ (cf. \cite[Definition 10]{LejmiGabor}). In the definition of uniform J-stability we use one of several equivalent `norms' of test configurations, privileging the so called minimum norm of \cite{DervanUKs}, which coincides with the non-archimedean $I-J$ functional, and we denote by $||.||$. This norm is related to the $\mathrm{E}_{\theta,\omega}^{\mathrm{NA}}(\Phi)$ functional via the identity
$$
||\Phi|| = \mathrm{E}_{\omega,\omega}^{\mathrm{NA}}(\Phi)
$$
taking $[\theta] = [\omega]$. For readability purposes we will occasionally suppress the double indices, writing $\mathrm{E}_{\theta}^{\mathrm{NA}}(\Phi)$ and $\mathrm{E}_{\omega}^{\mathrm{NA}}(\Phi)$ for the above quantities. \\

We further recall the special class of examples of test configurations, given by deformation to the normal cone, which we will make use of in Section \ref{Section J}: Given a curve $C \subseteq X$, there is a natural way to associate a test configuration $\Phi_{C,\kappa}$ to it via the \emph{deformation to the normal cone} construction introduced by Ross-Thomas \cite{RossThomas}. Its total space is given by $\mathcal{X} := \mathrm{Bl}_{C \times \left\{0\right\}}(X \times \mathbb{P}^1)$ and the associated relatively K\"ahler class is $\mathcal{A}_{C,\kappa} := \pi^*\alpha - \kappa[E]$, where $E$ is the exceptional divisor and $\pi: \mathcal{X} \rightarrow X$ is the map factoring through the blowdown map and the $1^{st}$ projection $p_1: X \times \mathbb{P}^1 \rightarrow X$. The test configuration $\Phi_{C,\kappa}$ is moreover said to be relatively K\"ahler precisely when $\kappa \in (0,\overline{\kappa}_C)$, where $$\overline{\kappa}_C := \sup \left\{ r > 0: \pi^*\alpha - r[E] > 0\right\}$$ is the Seshadri constant. 
The relevant non-archimedean functionals that we wish to consider then take particularly simple polynomial expressions along $\Phi_{C,\kappa}$, varying $\kappa$ in the interval $(0,\overline{\kappa}_C$), via the following well-known lemma:

\begin{lemma} \label{Lemma polynomial expansion} \emph{(\cite[Proposition 13]{LejmiGabor})}
Suppose $X$ is a smooth projective surface. For slope test configurations $\Phi_{C,\kappa} \in \mathcal{H}^{\mathrm{NA}}(X,[\omega])$ we have
$$
\mathbf{E}_{\theta,\omega}^{\mathrm{NA}}(\Phi_{C,\kappa}) = A_1\kappa^2 + B_1\kappa^3,
$$
$$
||\Phi_{C,\kappa}|| = A_2\kappa^2 + B_2\kappa^3,
$$
where $A_1,A_2,B_1,B_2$ are real numbers satisfying 
$$
A_1 = c_{\theta,\omega}\int_C \omega - \int_C \theta, \; \; B_1 = - \left(\frac{2[\theta].[\omega]}{3[\omega]^2} \right) C^2, \; \; A_2 =  \int_C \omega > 0, \; \; B_2 = -\frac{2}{3}C^2.$$
\end{lemma}

\begin{remark}
The proof of this result follows from standard intersection theory when $[\theta]$ and $[\omega]$ are rational classes. It also extends to $[\theta],[\omega] \in H^{1,1}(X,\mathbb{R})$ by continuity. In the purely transcendental setting this is related to the study of K\"ahler slope stability in \cite{Jacopo}.
\end{remark}

\noindent Finally, to fix notation and sign convention, we say that $(X,[\omega])$ is J-semistable (resp. J-stable) if 
$$
\mathrm{E}^{\mathrm{NA}}_{\theta,\omega}(\Phi) \geq 0
$$
(resp. $> 0$)
for all normal and relatively K\"ahler test configurations $\Phi$ for $(X,[\omega])$. We say that $(X,[\omega])$ is uniformly J-stable if there exists a $\delta > 0$ such that
$$
\mathrm{E}^{\mathrm{NA}}_{\theta,\omega}(\Phi) \geq \delta||\Phi||
$$
for all $\Phi \in \mathcal{H}^{\mathrm{NA}}(X,[\omega])$. We moreover say that a test configuration $\Phi_0$ is \emph{optimally destabilizing with respect to J-stability} if it is non-trivial and
\begin{equation} \label{Eq opt dest tc}
\frac{\mathrm{E}^{\mathrm{NA}}_{\theta,\omega}(\Phi_0)}{||\Phi_0||} = \inf_{\Phi \in \mathcal{H}^{\mathrm{NA}}(X,[\omega])}\frac{\mathrm{E}^{\mathrm{NA}}_{\theta,\omega}(\Phi)}{||\Phi||} \leq 0.
\end{equation}
The only difference between the above and the definition used in e.g. \cite{Gabortoric}, based on Donaldson's lower bound of the Calabi functional \cite{DonaldsonCalabi}, is the use of the minimum norm $||.||$ instead of the $L^2$ norm. 

For background on test configurations we refer to \cite{Tian, Donaldsontoric, SD1, DervanRoss, BHJ1, Zakthesis} and references therein. For variational methods and stability for the J-equation we refer to \cite{CollinsGabor}. In the case of the dHYM equation similar techniques have been developed and studied in \cite{JacobYau, CollinsJacobYau, ChuCollinsLee} and followup work. For the Z-critical equation see \cite{DMS,Dervan,McCarthy}. 

\medskip

\section{Proof of the main theorem for deformed Hermitian Yang-Mills and the Z-critical equations} \label{Section main thms}

\noindent In this section we prove the main theorems for the deformed Hermitian Yang-Mills and Z-critical equations, and map out a number of consequences of these results. 

\subsection{A meta proposition} 

\noindent As a key preparation for the proof of our main result, we make the following rather general observation: 

\begin{proposition} [Meta proposition] \label{Prop meta} 
Let $X$ be a compact Kähler surface with $\omega$ a Kähler form. Let $\tau \in \mathcal B_X \subseteq H^{1,1}(X,\mathbb R)$ be any big cohomology class that is not Kähler. Then, there exists a (non-empty) finite set of irreducible curves $\left\{E_1, \cdots, E_\ell\right\}$ such that
\begin{equation} \label{Eq 123}
\int_{E} \tau \leq 0 \iff E = E_i \textrm{ for some } i=1,\cdots, \ell.
\end{equation}
Moreover, the intersection matrix $(E_i\cdot E_j)$ of the curves $E_i$ is negative-definite and their classes $[E_i] \in H^{1,1}(X,\mathbb R)$ are linearly independent. In particular, $\ell \leq \rho(X)\leq h^{1,1}(X)$ where $\rho(X) = \dim_\mathbb R NS(X)_\mathbb R$. Moreover, if $X$ is projective, then $\ell \leq \rho(X) -1.$
\end{proposition}

\begin{proof}

\noindent Let us first assume that we are in the special case whereby the big class $\tau$ is not nef. The Zariski decomposition of $\tau$ can then be written as $$\tau = Z(\tau) + N(\tau) $$ where $Z(\tau)$ is a nef class and $$N(\tau) = \sum_i^s a_i [E_i]$$ for a unique non-zero effective $\RR$-divisor $$D = \sum_i^s a_i E_i$$ (with prime components $E_j \subseteq X$ and $a_j > 0$) such that the intersection matrix $(E_i \cdot E_j)$ is negative-definite and the classes $[E_i]$ are linearly independent in $NS(X)_\mathbb R \subseteq H^{1,1}(X,\mathbb R)$. Let $E \subseteq X$ be any closed irreducible curve such that $\int_E \tau < 0$. But if $E \neq E_j$ for any $j = 1, \cdots, s$ then we have $E \cdot E_j \geq 0$ because distinct irreducible curves always have a nonnegative intersection. Therefore, we get that $$\int_E \tau = \int_E Z(\tau) + \sum_{i = 1}^s a_i E\cdot E_i \geq \int_E Z(\tau) \geq 0 $$  which contradicts our assumption on $E$. So we must have that $E = E_j$ for some $j = 1, \cdots, s$, and since the classes $[E_i]$ are linearly independent in $NS(X)_\mathbb R \subseteq H^{1,1}(X,\mathbb R)$, we also have $s \leq \rho(X)$. Moreover, if $X$ is projective, then $NS(X)_\RR$ contains at least one positive eigenvector of the intersection form, so in that case $\ell \leq \rho(X)-1$. To summarise, we have proven that if $\tau$ is not nef, then there exist at most $\rho(X)$ (and in case $X$ is projective, at most $\rho(X) -1$) distinct curves $E\subseteq X$ such that $\int_E \tau < 0$, all of which occur as prime components of $N(\tau)$. 

Now suppose, in full generality of the proposition, that $\tau$ is not Kähler. Then, $\tau - \varepsilon[\omega]$ is not nef but still contained in the big cone $\mathcal B_X$ for any $\varepsilon > 0$ small enough. Now if $\int_E \tau \leq 0$ for some curve $E \subseteq X$, we  have $\int_E (\tau - \varepsilon[\omega]) < 0$, so the set of such curves is contained among the prime components of $N(\tau - \varepsilon[\omega])$, say $E_1, \cdots, E_s$. After relabelling if necessary, we may suppose $\int_{E_j} \tau \leq 0$ precisely for $j = 1, \cdots, \ell$. Then clearly we have $\ell \leq s \leq \rho(X)$ (and if $X$ is projective, $s \leq \ell \leq \rho(X)-1$). Finally, the claim about the intersection matrix follows because any submatrix of a negative-definite matrix is itself negative-definite.
\end{proof}

\medskip

\subsection{Global finiteness of $\mathcal{S}_{\tau}$  in compact sets}

\noindent For the sequel we introduce the following terminology:

\begin{definition}
\label{def sets of destabilizers}
Let $X$ be a compact Kähler surface with $\omega$ a Kähler form on $X$. Let $\tau \in \mathcal B_X$. 
\begin{enumerate}
    \item The \emph{set of destabilizing curves with respect to $\tau$} is the finite set $$\mathcal D_\tau := \left\{ E \subseteq X \textrm{ irreducible curve } | \ \int_E \tau \leq 0 \right\}. $$
    \item A set $\mathcal S$ of irreducible curves is a set of \emph{candidate destabilizers} for $\tau$ if $\mathcal D_\tau \subseteq \mathcal S$. 
    \item The set of \emph{Zariski negative curves for $\tau$} is the finite set $$ \mathrm{Neg}(\tau) := \bigcap_{\varepsilon > 0} \left\{E \subseteq X \textrm{ irreducible curve } | \  E \textrm{ is a prime component of } N(\tau - \varepsilon[\omega])\right\}.$$
\end{enumerate} 
\end{definition}

\begin{remark}
\begin{enumerate}\label{Rmk Discussion Meta prop}
    \item {Note that $\operatorname{Neg}(\tau)$ does not depend on the Kähler class $[\omega]$ we choose.}
    \item It is clear from the proof of the Proposition \ref{Prop meta} that $\mathcal D_{\tau} \subseteq  \mathrm{Neg}(\tau)$, and hence $\mathcal S_\tau := \mathrm{Neg}(\tau)$ always provides a natural set of candidate destabilizers for $\tau$. 
    \item For any big class $\tau$, the closed positive current $N(\tau)$ appearing in the Zariski decomposition is the divisor part in the Siu decomposition of a current of minimal singularities (see \cite{Boucksomthesis}). 
    \item 
    In previous literature, a slightly weaker analogue of Proposition \ref{Prop meta} appeared in \cite[Proposition 4.5]{SongWeinkove}, which uses the Siu decomposition of closed positive (1,1)  currents and a regularisation result due to Demailly (see the proof of \cite[Proposition 4.5]{SongWeinkove} as well as \cite{Lamari} and references therein). The Zariski decomposition has the added advantage that the negative part 
    is precisely given by the Siu decomposition of a current of \emph{minimal singularities} (see \cite[Section 3.2]{Boucksomthesis} for reference). 
\end{enumerate}
\end{remark}

\noindent To obtain a finite number of test conditions for the various geometric PDE that can be used also under small perturbations of the underlying cohomological data, it is essential to ask how the sets $\mathrm{Neg}(\tau)$ change as $\tau$ varies across $\mathcal P^+_X$. In particular, if a suitable subset $K \subseteq \mathcal P^+_X$ is specified, one hopes to obtain a `global set of test curves' $\mathcal V_K$ against which we can test any given class $\tau \in K$. 
When $K$ is a compact subset of $\mathcal P^+_X$, the following result produces a finite set of curves that is sufficient uniformly across all of $K$, and moreover gives an explicit upper bound on its cardinality:

\begin{lemma} \label{Lem cardinality bound cvx hull}
Let $X$ and $\omega$ be as in Proposition \ref{Prop meta}, and let $K \subseteq \mathcal B_X$ be any subset contained in the positive cone over the strict convex hull of finitely many pseudoeffective classes $\tau_1,\dots,\tau_k$ on $X$. Then $$ \mathcal S_K := \bigcup_{\tau \in K} \mathcal D_{\tau}$$ is a finite set of curves of negative self-intersection, of cardinality $|\mathcal{S}_K| \leq k\rho(X)$, and is a set of candidate destabilizers for every $\tau \in K$.
If $X$ is projective, then the cardinality of $\mathcal S_K$ does not exceed $k\rho(X) - k$.
\end{lemma}
\begin{proof}
Suppose that 
$
\tau = \sum_{i = 1}^k a_i\tau_i, \; \; a_i > 0.
$
Then $\int_C \tau \leq 0$ implies $\int_C \tau_i\leq 0$ for some $i$. Hence $C$ is in the set of destabilizers
$$
\mathcal{D}_{\tau_i} \subseteq \bigcup_i \mathcal{D}_{\tau_i} 
$$
which is of cardinality at most $k \max_i|\mathcal{D}_{\tau_i}| \leq k\rho(X)$ (or $\leq k\rho(X) - k$ if $X$ is projective), by Proposition \ref{Prop meta}. 
\end{proof}

\noindent For future use we record the elementary fact that the above in particular applies if $K$ is a compact subset of $\mathcal{P}_X^+$: 

\begin{lemma} \label{Lemma convex hull}
Any compact subset $K \subseteq \mathcal{P}_X^+$ is contained in the convex hull of finitely many points in $\mathcal{P}_X^+$.
\end{lemma}

\begin{proof}
Without loss of generality, assume that $K$ is connected. We may then choose an open covering of $K$ by means of open cubes $C_{\nu}$ whose closure is in $\mathcal{P}_X^+$. By compactness there is a finite open subcover 
$$
\bigcup_{\nu \in I, |I| < +\infty} C_{\nu}.
$$
Taking the convex hull of enough of the (finitely many) vertices of $C_{\nu}$, $\nu \in I$, this set clearly contains $K$. 
\end{proof}

\begin{remark} \label{Cor Nakai Moishezon} It is moreover worth noting that the above result has an interesting interpretation already in the case of the Nakai-Moishezon criterion, which is used to test if a given class $\alpha \in H^{1,1}(X,\mathbb{R})$ is a K\"ahler class (underlining that this is a finite problem). Indeed, let $(X,\omega)$ be a compact K\"ahler surface and $K \subseteq \mathcal{P}_X^+$ a compact set. Then there is a finite number of curves $E_1,\cdots, E_\ell$ on $X$ \emph{(depending only on $K$)} such that, for any $\alpha \in K$, the following are equivalent:
\begin{enumerate}
    \item The class $\alpha$ is K\"ahler.
    \item We have
    $$
    \int_{E_i} \alpha > 0
    $$
    for all $i \in \left\{1,\dots,\ell\right\}$. 
\end{enumerate}
As before the curves $E_i$ have negative self-intersection, and if $K$ is contained in the convex hull of $k$ pseudoeffective classes, then $\ell \leq k\rho(X)$ (and $\ell \leq k\rho(X) - k$ if $X$ is projective).
\end{remark}

\medskip

\subsection{Proof of the Main Theorem for the deformed Hermitian-Yang-Mills equation}
We now turn to the proof of Theorem \ref{Thm main dHYM intro}. To fix notation, let us suppose we are given a pair of Kähler classes $\beta, \alpha\in\mathcal C_X$ and a Kähler form $\theta$ in $\beta$. Then, the topological phase $\hat\Theta(\beta,\alpha)$ defined (modulo $2\pi$) by 
$$
\int_X\operatorname{Im}\left(e^{-\i\hat\Theta(\beta,\alpha)}(\beta+\i\alpha)^2\right)= 0
$$
satisfies 
$$
\cot(\hat\Theta(\beta,\alpha)) = \frac{\int_X(\beta^2-\alpha^2)}{2\int_X \alpha\cdot\beta}.
$$
Indeed, this follows straightforwardly from the readily apparent identity 
$$
\operatorname{Im}\left(\left(\cos(\hat\Theta(\beta,\alpha))-\i\sin(\hat\Theta(\beta,\alpha))\right)\left(\int_X(\beta^2 - \alpha^2) + \i\int_X 2\alpha\cdot\beta\right)\right)=0.
$$
Using this, we see that the deformed Hermitian Yang-Mills equation 
$$
\operatorname{Im}\left(e^{-\i\hat\Theta(\beta,\alpha)}(\theta+\i\omega)^2\right)=0,
$$
is equivalent to solving the equation 
$$
2\cos(\hat\Theta(\beta,\alpha))\theta\wedge\omega - \sin(\hat\Theta(\beta,\alpha))(\theta^2 - \omega^2) = 0
$$
for a smooth (1,1) form $\omega \in \alpha$. Upon re-arranging and completing the square in $\omega$, we obtain
$$
(\omega+ \cot(\hat\Theta(\beta,\alpha))\theta)^2 = \left(1 + \cot(\hat\Theta(\beta,\alpha))^2\right)\theta^2.
$$
which is a complex Monge-Ampère equation for the class $\alpha+\cot(\hat\Theta(\beta,\alpha))\beta$. Thus, in this subsection, we shall write 
$$
\tau(\beta,\alpha):= \alpha + \cot(\hat\Theta(\beta,\alpha))\beta.
$$
In this notation, Theorem \ref{Thm main dHYM intro} can be restated as follows.
\begin{theorem}\label{Thm Main dHYM Proof} Let $X$ be a compact Kähler surface and let $K\subseteq \mathcal C_X \times \mathcal C_X$ be a compact subset. Then, there exists a non-negative integer $\ell \geq 0$ and curves of negative self-intersection $E_1,\cdots, E_\ell$ on $X$ (depending only on $K$) such that for all $(\beta,\alpha)\in K$ the following are equivalent. 
\begin{enumerate}
    \item For any choice of Kähler metric $\theta$ in $\beta$, there exists a smooth (1,1) form $\omega$ which is a solution to the deformed Hermitian Yang-Mills equation 
    $$
    \operatorname{Im}\left(e^{-\i\hat\Theta(\beta,\alpha)}(\theta+\i\omega)^2\right)=0,
    $$
    and the solution $\omega$ is a Kähler form in $\alpha$ if $\int_X(\alpha^2 -\beta^2)>0.$
    \item For every curve $E\subseteq X$, we have 
    $$
    \int_{E}\tau(\beta,\alpha) >0.
    $$
    \item For $i = 1,\cdots, \ell$, we have 
    $$
    \int_{E_i}\tau(\beta,\alpha) >0.
    $$
\end{enumerate}
    
\end{theorem}
\begin{proof}
    From the discussion preceding the statement of the Theorem, we see that 
    $$
    \int_X\tau(\beta,\alpha)^2 = (1+\cot^2(\hat\Theta(\beta,\alpha)))\int_X \beta^2 > 0.
    $$
    Moreover, observe that
    $$
    \int_X \tau(\beta,\alpha)\cdot\alpha = \frac{1}{2}\int_X(\beta^2 + \alpha^2) > 0.
    $$
    and so 
    $$
    \tau(\beta,\alpha) \in \mathcal P^+_X.
    $$
    Thus, the continuous map 
    $$\mathcal C_X \times \mathcal C_X \to H^{1,1}(X,\RR), \; \; (\beta,\alpha)\mapsto \tau(\beta,\alpha)
    $$
    takes $K$ onto a compact subset $\tilde K$ of $\mathcal P^+_X$. Now, by Lemma \ref{Lemma convex hull}, $\tilde K$ is contained in the convex hull of finitely many points of $\mathcal P^+_X$ and by Lemma \ref{Lem cardinality bound cvx hull}, there exists a non-negative integer $\ell \geq 0$ and finitely many curves of negative self-intersection $E_1,\cdots,E_\ell$ such that a class $\tau \in \tilde K$ is Kähler if and only 
    $$
    \int_{E_i} \tau > 0
    $$
    for $i= 1, \cdots, \ell.$ The proof that these curves satisfy the conclusion of the Theorem now proceeds via a standard argument whereby we reduce the equation to a complex Monge-Ampère equation. For the sake of completeness, we recall this simple argument.\\
    
    \noindent Suppose $(\beta,\alpha)\in K$ is such that the deformed Hermitian Yang-Mills equation admits a smooth solution $\omega$ for a choice of Kähler form $\theta \in \beta$. Then, from the above discussion, we see that 
    $$
    (\omega+\cot(\hat\Theta(\beta,\alpha)\theta)^2 = (1+\cot^2(\hat\Theta(\beta,\alpha)))\theta^2. 
    $$
    From this equality of forms, it follows that the (1,1) form $(\omega + \cot(\hat\Theta(\beta,\alpha)))$ is Kähler. Indeed, the equality obviously implies that $\omega + \cot(\hat\Theta(\beta,\alpha))\theta$ is non-degenerate and defines the same orientation as the Kähler form $\theta$. On a surface, this condition is equivalent to definiteness. But the topological fact that the class $\tau(\beta,\alpha)$ is in $\mathcal P^+_X$ ensures that the form $\omega + \cot(\hat\Theta(\beta,\alpha)$ is positive-definite. This shows that the class $\tau(\beta,\alpha)$ is a Kähler class, and proves that (1) implies (2). 
    
    It is obvious that (2) implies (3). 
    Finally, let there be given any choice of Kähler metric $\theta \in \beta$. Then (3) implies, by our choice of the curves $E_i$, that the class $\tau(\beta,\alpha)$ is a Kähler class. By Yau's solution of the Calabi conjecture \cite{Yau}, this implies that we can find a Kähler metric $\chi \in \tau(\beta,\alpha)=\alpha + \cot(\hat\Theta(\beta,\alpha))\beta$ such that 
    $$
    \chi^2 = (1  + \cot^2(\hat\Theta(\beta,\alpha)))\theta^2.
    $$
    But then $\omega = \chi - \cot(\hat\Theta(\beta,\alpha))\theta$ satisfies the deformed Hermitian Yang-Mills equation. This shows that (3) implies (1). Finally, note that if $\int_X(\alpha^2-\beta^2)>0$ then clearly $\cot(\hat\Theta(\beta,\alpha))<0$ and so the form $-\cot(\hat\Theta(\beta,\alpha))\theta$ is a Kähler class, and therefore, so is $\omega = \chi -\cot(\hat\Theta(\beta,\alpha))\theta$.
    
\end{proof}

\begin{remark}\label{Rmk dHYM phase}
The hypothesis $\int_X (\alpha^2-\beta^2) > 0$ is usually formulated as the phase hypothesis $\hat\Theta(\beta,\alpha)\in (\frac{\pi}{2},\pi)$, and is called the \emph{supercritical regime}.
\end{remark}

\begin{corollary}
    Suppose $X$ is a compact Kähler surface that admits no curves of negative self-intersection. Then, the deformed Hermitian Yang-Mills equation always admits a solution for any pair of Kähler classes $\beta,\alpha$ and any choice of Kähler metric $\theta \in \beta$.
\end{corollary}
\begin{proof}
    Take $K = \left\{(\beta,\alpha)\right\}$ in \ref{Thm Main dHYM Proof} and observe that we must have $\ell = 0$, so the third condition is vacuously satisfied. 
\end{proof}
\begin{remark}
    A version of the above Corollary can be derived (with slightly different hypotheses) from \cite[Theorem 1.4]{FuYauZhang}, although in their work the emphasis is more on the study of a dHYM flow. We also remark that the above Theorem \ref{Thm Main dHYM Proof} is a special case of Theorem \ref{Thm Main Z Critical Proof} below.
\end{remark}

\begin{proof}[Proof of Corollary \ref{Cor dHYM structure intro}]
{The only claim that requires justification is the claim about the boundary $\partial \mathcal U|_K$ being of real codimension one, the rest of the claims following easily from the proof of Theorem \ref{Thm Main dHYM Proof}. To justify this last claim, we must show that any given `wall' defining $\partial \mathcal U$ is a real algebraic submanifold of $\mathcal C_X \times \mathcal C_X$ of codimension one. More precisely, given any curve $E$ appearing in Theorem \ref{Thm main dHYM intro}, the locus 
$$
W_E := \left\{ (\alpha,\beta)\in \mathcal C_X \times \mathcal C_X \ | \  \int_E \tau_{\mathrm{dHYM}}(\alpha,\beta) = 0\right\}
$$
is a real algebraic submanifold of codimension one. The fact that it is real algebraic is straightforward, since it is the zero locus of the composition of the real algebraic maps
$$
(\alpha,\beta)\mapsto \tau_{\mathrm{dHYM}}(\alpha,\beta)=\alpha + \frac{\beta^2-\alpha^2}{2\beta\cdot \alpha}\beta, \quad \quad \tau \mapsto \int_E \tau.
$$
(The first map is algebraic on the open set $ \left\{(\alpha,\beta)\in H^{1,1}(X,\mathbb R) \ | \  2\beta\cdot\alpha \neq 0\right\}$ which certainly contains $\mathcal C_X \times \mathcal C_X$.) 
Therefore, it suffices to prove that 0 is a regular value of the map
$$
(\alpha,\beta) \mapsto 
\int_E \tau_{\mathrm{dHYM}}(\alpha,\beta).$$ To this end, assume $(\alpha,\beta)$ is a zero of the above map and define the function 
$$
f(\varepsilon) := \int_E \left(\alpha + \varepsilon \beta + \frac{\beta^2 - (\alpha + \varepsilon\beta)^2}{2\beta\cdot (\alpha + \varepsilon \beta)}\beta\right).
$$ Clearly, $f(0) = 0$ and a straightforward calculation shows that
$$
f^\prime(\varepsilon) = \int_E\left( \beta + \frac{(-2\alpha_\varepsilon\cdot\beta)(2\beta\cdot\alpha_\varepsilon)-(\beta^2 - \alpha_\varepsilon^2)(2\beta^2)}{(2\beta\cdot\alpha_\varepsilon)^2}\beta\right)=\frac{-\beta^2}{\beta\cdot\alpha_\varepsilon} \int_E \left(\frac{\beta^2 - \alpha_\varepsilon^2}{2\beta\cdot\alpha_\varepsilon}\beta\right)
$$
where we have written $\alpha_\varepsilon = \alpha + \varepsilon\beta$. This means that
$$
f^\prime(0) = \frac{-\beta^2}{\beta\cdot\alpha}\int_E \frac{\beta^2 - \alpha^2}{2\beta\cdot\alpha}\beta = \frac{\beta^2}{\beta\cdot \alpha}\int_E \alpha \neq 0,
$$
where the last equality follows from the fact that $f(0) = 0$ by assumption.}
\end{proof}
\subsection{Proof of the Main Theorem for the Z-critical equation}

In this section we prove the main Theorem \ref{Thm Main Theorem Z Critical Equation} about the Z-critical equation, whose statement we detail below. Before embarking on the proof, we fix some notation: let $\Omega = (\beta,\rho,U)$ be the data of a \emph{polynomial central charge} on a compact K\"ahler surface, i.e. a polynomial $\rho = \rho(t) = \rho_0 + \rho_1 t + \rho_2 t^2$ a with non-zero complex coefficients subject to the constraints 
\begin{equation}\label{Eqn stab vector phases}
\operatorname{Im}(\rho_2)>0, \operatorname{Im}\left(\frac{\rho_1}{\rho_2}\right)>0, \operatorname{Im}\left(\frac{\rho_0}{\rho_1}\right)> 0,  
\end{equation}
Let $\beta$ and $U$ be cohomology classes with $\beta \in\mathcal C_X$ a K\"ahler class, and $U = 1 + U_1 + U_2 \in \oplus_i H^{i,i}(X,\RR)$ a unipotent cohomology class (with graded components $U_i \in H^{i,i}(X,\RR)$). Then, for any holomorphic line bundle $L \to X$, we shall write $$Z_\Omega(L):= \int_X \rho(\beta)\cdot U \cdot \operatorname{ch}(L)$$ where $\operatorname{ch}(L) = 1 + c_1(L) + c_1(L)^2/2 \in \oplus_i H^{i,i}(X,\RR)$ denotes the Chern character of $L$. We shall moreover always assume that our choice of $\Omega$ is such that $Z_\Omega(L)$ lies in the upper-half plane for the holomorphic line bundle $L$ under consideration. (In this case, we shall informally say that \emph{$\Omega$ defines a polynomial central charge.} This can always be achieved, for example, by scaling $\theta \mapsto t\theta$ for $t>0$ very large for any fixed $L$.) Then, the \emph{phase} or \emph{$Z_\Omega$-phase} $\varphi(L)$ of $L$ (with respect to $Z = Z_\Omega$) is the real number $$\varphi(L):= \arg Z_\Omega(L). \\$$

\noindent The $Z_\Omega$-critical equation is specified once a \emph{lift} $\tilde\Omega$ of $\Omega$ is fixed. Concretely, this means that we fix a choice of K\"ahler metric $\theta \in \beta$ and smooth representative $\tilde U_i \in U_i$. Then we define $\tilde Z_\Omega(L,\tilde \Omega, h)$ as the  degree (2,2) part of the form $$ \rho(\theta) \wedge \tilde U \wedge \operatorname{ch}(L,h)$$ where $\tilde U = 1 + \tilde U_1 + \tilde U_2$ and $\operatorname{ch}(L,h)$ is the Chern-Weil representative of $\operatorname{ch}(L)$ with respect to a Hermitian metric $h$; namely $$ \operatorname{ch}(L,h) = \exp\left(\frac{\i}{2\pi}F_h\right) = 1 + \frac{\i}{2\pi}F_h+ \frac{1}{2}\left(\frac{\i}{2\pi}F_h\right)^2 $$ with $F_h$ the curvature (1,1)-form of the Chern connection associated to $h$. Then the $Z_\Omega$-critical equation takes the form 
\begin{equation}\label{Eqn Z critical}
\operatorname{Im}\left(e^{-\i\varphi(L)}\tilde Z_\Omega(L,\tilde \Omega,h)\right) = 0.
\end{equation} This is a second order fully non-linear equation for the metric $h$.\\

In \cite[Section 2.3]{DMS} the authors derive the notion of a subsolution for the Z-critical equation. (See \cite[Definition 2.33]{DMS}.) They then prove that in our present special case where $X$ is a smooth projective surface, $E = L \to X$ is a line bundle, and the lifted data $(\theta,\rho, \tilde U)$ satisfy the so-called \emph{volume form hypothesis}, the existence of a subsolution is equivalent to the existence of a solution. More precisely, define the forms $\tilde\eta = \tilde\eta(\tilde\Omega,L)$ and $\tilde\gamma=\tilde\gamma(\tilde\Omega,L)$ by the formula 
\begin{equation}\label{Eqn eta gamma def}
\operatorname{Im}\left(e^{-\i\varphi(L)}\tilde Z_\Omega(L,h)\right) = c(\chi_h^2 + \chi_h\wedge\tilde\eta + \tilde\gamma) 
\end{equation} 
where $c \in \RR$ is a (non-zero) normalisation constant and $\chi_h = \frac{\i}{2\pi}F_h$ is the curvature form of (the Chern connection associated to) the Hermitian metric $h$. (This formula implicitly assumes that $\varphi(L)\neq \arg(\pm\rho_0)$ so the $\chi_h^2$ term does not vanish in the expansion. See the proof of Lemma \ref{Lemma eta gamma formulas}.) As the authors observe in \cite{DMS} after completing the square, the equation is equivalent to solving 
$$
\left(\chi_h+\frac{1}{2}\tilde\eta\right)^2 = \frac{1}{4}\tilde\eta^2 - \tilde\gamma.
$$ 
Then the (lifted) data $\tilde\Omega = (\theta, \rho, \tilde U)$ are said to satisfy the \emph{volume form hypothesis for L} if the (2,2)-form $$ \frac{1}{4}\tilde\eta^2 - \tilde\gamma \in \mathcal A^{2,2}(X,\RR)$$ is a volume form on $X$. \\

\begin{lemma}\label{Lemma eta gamma formulas}
    Let $\Omega = (\beta,\rho,U)$ be a choice of stability data defining a polynomial central charge on a projective surface $X$, with a fixed lift 
    $$
    \tilde\Omega = \left(\theta,\rho_0+\rho_1 t+\rho_2 t^2, 1 + \tilde U_1 + \tilde U_2\right),
    $$ 
    and let $L\to X$ be a holomorphic line bundle on $X$ such that $\varphi(L) \neq \arg(\pm\rho_0)$. Then the forms $\tilde\eta(\tilde \Omega,L)$ and $\tilde\gamma(\tilde\Omega,L)$ are given by  
    \begin{align}\label{Eqn eta gamma formulas}
    \tilde\eta(\tilde\Omega,L) &= \frac{2}{c_0}\left(c_0\tilde U_1 + c_1\theta\right),\\
    \tilde\gamma(\tilde\Omega,L) &= \frac{2}{c_0}\left(c_0\tilde U_2 + c_1\theta\wedge\tilde U_1 +c_2\theta^2\right),
    \end{align}
    where $c_k = \mathrm{Im}(\rho_k)\cot\varphi(L)-\mathrm{Re}(\rho_k)$.
\end{lemma}
\begin{proof}
    This is a straightforward computation. Writing $\chi_h$ for the curvature (1,1) form $\frac{\i}{2\pi}F_h$, we first note that $Z_{\tilde\Omega}(L,h)$, i.e. the (2,2) part of the form 
    $$
    \left(\rho_0+\rho_1 \theta + \rho_2 \theta^2 \right)\wedge \mathrm{ch}(L,h) \wedge \left(1 + \tilde U_1 +\tilde U_2 \right),
    $$
    is given by 
    $$
    \frac{1}{2}\rho_0 \chi_h^2 + \left(\rho_0 \tilde U_1 + \rho_1 \theta \right)\wedge \chi_h + \rho_0 \tilde U_2 +\rho_1 \theta\wedge\tilde U_1 +\rho_2 \theta^2.
    $$
    Let us write $e^{\i\varphi(L)} = A + \i B$ where $A,B$ are real numbers. By the hypothesis that $\Omega$ defines a polynomial central charge, $B > 0$. Therefore, we can write $e^{-\i\varphi(L)} = B^{-1}(\cot\varphi(L) - \i)$. Now observe that 
    $$
    \operatorname{Im}(e^{-\i\varphi(L)}\tilde Z_{\tilde\Omega} (L,h)) = B^{-1}\left(\frac{1}{2} c_0 \chi_h^2 + (c_0 \tilde U_1 + c_1 \theta)\wedge \chi_h + c_0 \tilde U_2 + c_1 \theta\wedge\tilde U_1 + c_2 \theta^2\right)
    $$
    where 
    $
    c_k = \mathrm{Im}(\rho_k)\cot \varphi(L) - \mathrm{Re}(\rho_k).
    $
    Now, if $\varphi(L)\neq \arg(\pm\rho_0)$ then $c_0 \neq 0$ and the result follows by comparing with \eqref{Eqn eta gamma def}.
\end{proof}
Let us therefore set \begin{align}    
    \eta(\Omega,L) &:= \frac{2}{c_0}\left(c_0 U_1 + c_1\beta\right) \in H^{1,1}(X,\RR),\\
    \gamma(\Omega,L) &:= \frac{2}{c_0}\left(c_0 U_2 + c_1\beta\wedge U_1 +c_2\beta^2\right)\in H^{2,2}(X,\RR).
    \end{align} 

\begin{corollary}\label{Cor positive volume implies volume form}
    Let $X, L$ and $\Omega$ be as in the Lemma. If $$V(\Omega,L):= \int_X \frac{1}{4}\eta(\Omega,L)^2 - \gamma(\Omega,L) > 0$$ then there exists a choice of lift $\tilde\Omega$ that satisfies the volume form hypothesis for $L$. 
\end{corollary}
\begin{proof}
    Clearly, if the numerical inequality $V(\Omega,L)>0$ is satisfied, then the class $$\frac{1}{4}\eta(\Omega,L)^2 - \gamma(\Omega,L)\in H^4(X,\RR)$$ contains a volume form $v$. Fix any lift $\tilde\Omega_0 = (\theta,\rho,1+\tilde U_{1} + \tilde U_{2})$ of $\Omega$. Then, by the $\ddbar{}$-lemma, there exists a real valued (1,1) form $\zeta$ on $X$ such that $$v = \frac{1}{4}\tilde\eta(\tilde\Omega_0,L)^2 - \tilde\gamma(\tilde\Omega_0,L) + \i \ddbar{\zeta}.$$ Setting $\tilde U_1^\prime = \tilde U_{1}$ and $$\tilde U^\prime_2 = \tilde U_{2} - \frac{\i}{2}\ddbar{\zeta},$$ we see immediately from \eqref{Eqn eta gamma formulas} that if $\tilde\Omega = (\omega,\rho,1+\tilde U^\prime_1+\tilde U^\prime_2)$ then $$v = \frac{1}{4}\tilde\eta(\tilde\Omega,L)^2 - \tilde\gamma(\tilde\Omega,L)$$ is a volume form.
\end{proof}

In this notation, Theorem \ref{Thm main dHYM intro} can be stated as follows. 
\begin{theorem}\label{Thm Main Z Critical Proof}
    Let $X$ be a projective surface and $L\to X$ a holomorphic line bundle on $X$. Suppose $K \subseteq \mathcal C_X \times (\CC^*)^3\times  \bigoplus_i H^{i,i}(X,\RR)$ be a compact subset such that each $\Omega \in K$ defines a polynomial central charge $Z_\Omega$ on $X$. Morevoer, assume that for each $\Omega \in K$, we have $V(\Omega,L) > 0$ and $\varphi(L) \neq \arg(\pm\rho_0)$. Then, there exists a non-negative integer $\ell \geq 0$ and finitely many curves $E_1, \cdots, E_\ell$ on $X$, of negative self-intersection, such that the following are equivalent. \begin{enumerate}
        \item For every $\Omega \in K$ and every lift $\tilde\Omega$ satisfying the volume form hypothesis at $L$, the $Z_\Omega$-critical equation admits a solution. 
        \item For $i = 1, \cdots, \ell$, we have $$ s(\Omega,L)\left(\int_{E_i} c_1(L) + \frac{1}{2}\eta(\Omega,L) \right) > 0$$
        where $s(\Omega,L)$ is the sign of the nonzero real number 
        $$
        \int_{X}\left(c_1(L)+\frac{1}{2}\eta(\Omega,L)\right)\cdot\beta \in \RR.
        $$
    \end{enumerate}
\end{theorem}

\begin{remark}
{Note that for any line bundle $L$ the set of $\Omega$ satisfying $V(\Omega,L)>0$ and $\varphi_\Omega(L) \neq \arg(\rho_0)$ (the natural `phase' hypotheses in the context of the Z-critical equation) is non-empty and open. This can easily be seen by perturbing the stability vector $\rho$ (if required) and then scaling $\beta \mapsto k\beta$ for $k>0$ large.}
\end{remark}
\begin{proof}[Proof of Theorem \ref{Thm Main Z Critical Proof}]
Let a compact subset $K \subseteq \mathcal C_X \times (\mathbb C^*)^3 \times \bigoplus_i H^{i,i}(X,\mathbb R)$ be given such that each element $\Omega = (\beta,\rho,U) \in K$ defines a valid polynomial central charge with $\varphi_\Omega(L)\neq\arg(\pm\rho_0)$ and such that $V(\Omega,L) > 0$. Recall that the  topological constant $\varphi(L)= \varphi_\Omega(L)$ is chosen precisely so that $$ \int_X \operatorname{Im}\left(e^{-\i\varphi_\Omega (L)}Z_\Omega(L)\right) = 0.$$ In our notation, this is equivalent to $$0 = \int_X\left(c_1(L)^2 + c_1(L)\cdot \beta(\Omega,L) + \gamma(\Omega,L)\right) = \int_X\left(c_1(L) + \frac{1}{2}\eta(\Omega,L)\right)^2 - V(\Omega,L).$$ Therefore, the inequality $V(\Omega,L) >0$ implies that the class $\sigma(\Omega,L) := c_1(L) + \frac{1}{2}\eta(\Omega,L)$ has positive self-intersection. So, by the Hodge-Index Theorem, either $\sigma(\Omega,L)$ or its negative is in the cone $\mathcal P^{+}_X.$ Let $s(\Omega,L)\in \left\{\pm1\right\}$ be defined by the condition that $$\tau_Z(\Omega,L):= s(\Omega,L)\tau(\Omega,L) \in \mathcal P^{+}_X.$$ Clearly, we have $$s(\Omega, L) = \operatorname{sign}\left(\int_X \sigma(\Omega,L)\cdot\beta\right).$$ Thus, the map $$\mathcal C_X \times (\mathbb C^*)^3\times \bigoplus_i H^{i,i}(X,\RR)\to H^{1,1}(X,\RR), \; \; \Omega \mapsto \tau_Z(\Omega,L)$$ is continuous and takes $K$ onto a compact subset $\tilde K = \tau_Z(K)$ of $\mathcal P^+_X$. By Lemma \ref{Lemma convex hull}, $\tilde K$ is contained in the convex hull of finitely many points of $\mathcal P^+_X$. Now, by Lemma \ref{Lem cardinality bound cvx hull}, there exist finitely many curves $E_1, \cdots, E_\ell$ of negative self-intersection such that for each $\tau \in \tilde K$, $\tau$ is Kähler if and only if $$ \int_{E_i} \tau  >0 $$ for $i = 1, \cdots, \ell.$ Finally, according to \cite[Section 2.3.3]{DMS} this is equivalent to the existence of a subsolution for the $Z_\Omega$-critical equation, and by \cite[Theorem 2.45]{DMS}, equivalent to the existence of a solution. This completes the proof. 
\end{proof}

\begin{corollary}
    Let $X$ and $L$ be as in the Theorem \ref{Thm Main Z Critical Proof}. Suppose $X$ does not admit any curves of negative self-intersection. Let $\Omega$ be a choice of stability data defining a polynomial central charge $Z_\Omega$ and such that $\varphi(L)\neq\arg(\pm\rho_0)$ and $V(\Omega,L) > 0$. Then, for any lift $\tilde\Omega$ satisfying the volume form hypothesis, the $Z_\Omega$-critical equation admits a solution on $L$.
 \end{corollary}
 \begin{proof}
     In Theorem \ref{Thm Main Z Critical Proof}, take $K = \left\{\Omega\right\}$ and note that necessarily $\ell = 0$. Thus, the second of the two equivalent conditions is vacuously true.
 \end{proof}

\noindent We remark that in this setting a completely analogous result to Corollary \ref{Cor dHYM structure intro} also holds, {and is moreover in close analogy with results in the theory of Bridgeland stability. (See, for example, \cite[Proposition 9.3]{BridgelandK3}.)}
\begin{corollary} \label{Cor Z-critical structure}
    Fix a finite set $S$ of holomorphic line bundles $L_i\to X$ (for $i = 1, \cdots, k$) on a compact Kähler surface $X$. Let $V_S$ denote the set of triples $\Omega= (\beta,\rho,U)$ in $\mathcal C_X \times (\CC^*)^3 \times \bigoplus H^{i,i}(X,\RR)$ such that $Z_\Omega(L_i)$ lies in the upper half-plane, $\varphi_\Omega(L_i)\neq\arg(\pm\rho_0)$ and $V(\Omega,L_i) > 0$ for each $i = 1, \cdots, k$. Let $\mathcal U_i$ denote the subset of $V_S$ comprising those $\Omega$ such that for any choice of lift $\tilde{\Omega}$ of $\Omega$ satisfying the volume form hypothesis for $L_i$, the Z-critical equation \eqref{Eqn Z critical} admits a solution. Then, $\mathcal U_i$ is an open subset of $V_S$ and for any compact subset $K$ of $V_S$, the set $\mathcal U_i \cap K$ is cut out by the finitely many real algebraic inequalities
    $$
    W_{ij}(\Omega):= \int_{E_{ij}} \tau_Z(\Omega,L_i) > 0, \; \; j = 1, \cdots, \ell_i
    $$
    where $E_{i1}, \cdots, E_{i\ell_i}$ are the curves appearing in Theorem \ref{Thm Main Z Critical Proof}. {In particular, if $C$ is any connected component of 
    $$
    K \setminus \bigcup_{i,j} \left\{W_{ij}(\Omega)=0\right\}
    $$ then for any $i_0 \in \left\{1,\cdots, k\right\}$, we have $\Omega \in \mathcal U_{i_0}$ for some $\Omega \in C$ if and only if $\Omega \in \mathcal U_{i_0}$ for every $\Omega \in C$.  
    }
\end{corollary}
\begin{proof}
    The first claim is immediate from the proof of Theorem \ref{Thm Main Z Critical Proof}. The only thing that needs justification is the last sentence. But if $C$ is a connected component of 
    $$
    K \setminus \bigcup_{i,j} \left\{W_{ij}(\Omega)=0\right\}
    $$
    then, for each value of $i,j$, the continuous assignment $\Omega \mapsto W_{ij}(\Omega)$ must be non-zero and take the same sign for every $\Omega \in C$. Now fix a value $i_0$ for $i$ and note that for the finitely many values of $i,j$, the signs of $W_{ij}$ on $C$ are all positive or not all positive according as the class $\tau_Z(\Omega,L_{i_0})$ is Kähler or not, according as $\Omega \in \mathcal U_{i_0}$ or not, and this proves the claim.
\end{proof}

\noindent {Finally, to justify the terminology of `wall-chamber' decomposition that we have used we explain why the loci $$W_{ij}(\Omega) = 0$$ in the set $\mathcal U_i$ (in the notation of Corollary \ref{Cor Z-critical structure}) are real codimension one loci. More precisely, we have the following proposition.
\begin{proposition}\label{Cor Z crit codim 1}
Let $L$ be a holomorphic line bundle and let $V$ denote the set comprising stability data $\Omega = (\beta,\rho,U) \in \mathcal C_X \times (\mathbb C^*)^3 \times \bigoplus H^{i,i}(X,\mathbb R)$ for which $Z_\Omega (L)$ lies in the upper half-plane, $\varphi_\Omega (L) \neq \arg(\pm \rho_0)$ and $V(\Omega,L) > 0$. Let $E$ be any curve on $X$ such that $$\int_E \tau_Z(\Omega,L) = 0$$ for some $\Omega \in V$. Then the locus $$W_E = \left\{ \Omega \in V \ | \ \int_E \tau_Z (\Omega, L) = 0\right\}$$ is a real codimension one submanifold of $V$.
\end{proposition} \begin{proof}
    Recall that the class $\tau_Z(\Omega,L)$ is given, up to a sign, by $$\tau_Z = \pm\left(c_1(L) + \frac{1}{2}\eta(\Omega,L)\right),$$ where $$\eta(\Omega,L) = \frac{2}{c_0}\left(c_0U_1+c_1\beta\right)=2\left(U_1+\frac{c_1}{c_0}\beta\right)$$ and $c_k = c_k(\Omega, L) = \operatorname{Im}(\rho_k)\cot\varphi_{\Omega}(L) - \operatorname{Re}(\rho_k).$ Recall that $c_0 \neq 0$; this is a consequence of the hyphotesis that  $\varphi_\Omega(L)\neq \arg(\pm\rho_0)$. Fix $$\Omega = (\beta, \rho_0 + \rho_1 t + \rho_2 t^2,U)$$ such that $$\int_E \tau_Z(\Omega,L) = 0$$ for some curve $E$ and consider the family of stability data given by $$\Omega_\varepsilon := \left(\beta,\rho_0+\rho_1 t + \rho_2 t^2,U + \frac{\varepsilon}{\int_X \beta^2}\beta^2\right)$$ for $\varepsilon \in \mathbb R$ small. Then, we have $$Z_{\Omega_\varepsilon}(L) = Z_\Omega(L) + \frac{\varepsilon}{\int_X \beta^2}\int_X \beta^2 \cdot(\rho_0 + \rho_1 \beta + \rho_2 \beta^2)\cdot ch(L) = Z_\Omega(L) + \varepsilon \rho_0.$$ Writing $a:= \operatorname{Re}(\rho_0), b:= \operatorname{Im}(\rho_0)$, we get that $$\frac{d}{d\varepsilon}\Bigr|_{\varepsilon=0}\cot\varphi_{\Omega_\varepsilon}(L)= \frac{a\operatorname{Im}Z_\Omega(L) - b\operatorname{Re}Z_\Omega(L)}{(\operatorname{Im}Z_\Omega(L))^2}=-\frac{c_0}{\operatorname{Im}Z_\Omega(L)}$$ recalling that $c_0 = b \cot \varphi_\Omega(L) - a \neq 0$. (Throughout, we write $c_k = c_k(\Omega,L)$ for brevity, reserving the more elaborate notation $c_k(\Omega_\varepsilon,L)$ for the perturbed constants.) A simple calculation now shows that $$\frac{d}{d\varepsilon}\Bigr|_{\varepsilon=0}\frac{c_1(\Omega_\varepsilon,L)}{c_0(\Omega_\varepsilon,L)} = -\frac{1}{c_0 \operatorname{Im}Z_\Omega(L)}\left(c_0 \operatorname{Im}\rho_1 - c_1 \operatorname{Im}\rho_0\right) = -\frac{1}{c_0 \operatorname{Im}Z_\Omega(L)}\operatorname{Im}\left(\frac{\rho_0}{\rho_1}\right)\neq 0$$ the last inequality following from our hypotheses and the very definition of a stability datum (see \eqref{Eqn stab vector phases}). This shows that the function $$f(\varepsilon):= \int_E \tau_Z(\Omega_\varepsilon,L)$$ satisfies $$f^\prime(0) =\frac{d}{d\varepsilon}\Bigr|_{\varepsilon=0}\left( \int_E c_1(L) + U_1 + \frac{c_1(\Omega_\varepsilon,L)}{c_0(\Omega_\varepsilon,L)}\beta \right) = -\frac{1}{c_0 \operatorname{Im}Z_\Omega(L)}\operatorname{Im}\left(\frac{\rho_0}{\rho_1}\right)\int_E \beta \neq 0$$ and hence zero is a regular value of the map $$\Omega \mapsto \int_E \tau_Z(\Omega,L).$$
\end{proof}
\begin{remark}[Counter-example to globally finite wall-chamber decomposition] It is worth remarking that this wall-chamber decomposition cannot in general be \emph{globally} finite. Indeed, there exist Kähler surfaces which admit infinitely many distinct curve classes with negative self-intersection. (For example, the blowup of $\mathbb P^2$ in nine points in general position has infinitely many smooth rational curves of self-intersection $-1$.) If $X$ is any such surface, and $E$ is any curve on $X$ with negative self-intersection, then consider the family of stability data given by $$\Omega_r := \left(\beta, \frac{r}{E^2} \i - \frac{1}{\int_E \beta}t + \i t^2, 1 + [E] 
\right)$$ where $\beta$ is any Kähler class on $X$ with $\int_X \beta^2 = 1$ and $r > 0$ is a positive constant whose value we shall vary in the range $r \in (0, 2]$. (In particular, we have set $U_2 = 0$.) We wish to consider the Z-critical equations associated to this family of stability data on the \emph{trivial} line bundle $L = \mathcal O_X$. One verifies quite easily that $$Z_{\Omega_r}(L) = -1 + \i$$ independently of $r$ and therefore $\varphi_{\Omega_r}(L) = \frac{3\pi}{4} \neq \arg(\pm\frac{2}{3E^2}\i) = \pm\frac{\pi}{2}$. A straightforward calculation then shows that $$V(\Omega_r,L) = \frac{E^2(r-2)}{r}+\left(\frac{E^2}{r\int_E \beta}\right)^2 > 0$$ as $r\in (0,2].$ In other words, all the hypotheses of Theorem \ref{Thm Main Z Critical Proof} are satisfied. However, another straightforward calculation shows that $$\int_E \tau_Z(\Omega_r,L) = \frac{E^2(1-r)}{r},$$ so the assignment $\Omega_r \mapsto \int_E \tau_Z(\Omega_r,L)$ changes sign as $r$ crosses the value $r=1$. In other words, the stability data $\Omega_r$ cross the wall $$W_E = \left\{ \Omega \ | \ \int_E \tau_Z(\Omega,L) = 0 \right\}$$ defined by the curve $E$. Thus, every curve of negative self-intersection gives rise to a wall which has non-empty intersection with the space of admissible stability data for the trivial bundle, and there are therefore infinitely many distinct walls whenever there are infinitely many distinct curve classes with negative self-intersection.\end{remark}}

\bigskip

\section{Example of the wall-chamber decomposition: Blowup of $\mathbb{P}^2$ in two points} \label{Section example}

\noindent In this section, we present a simple illustration of the results from previous sections, in particular the wall-chamber decomposition of Corollary \ref{Cor Z-critical structure}. To this end, consider the blowup $\pi:X\to \mathbb P^2$ of $\mathbb P^2$ in two distinct points $p_1, p_2 \in \mathbb P^2$. Let $H$ denote the pullback of a line and $E_i = \pi^{-1}(p_i), i = 1,2$ the exceptional curves of the blowup. Then $H^{1,1}(X,\RR)$ is spanned by (the classes of) $H, E_1,E_2$, and it is easy to check that $X$ admits precisely three curves of negative self-intersection, namely $E_1, E_2$ and $T$ - the strict transform of the unique line passing through $p_1$ and $p_2$, whose class in cohomology is equal to $H-E_1-E_2$. (By the usual abuse of notation, we identify $E_i$ with $c_1(\mathcal O_X(E_i))$ etc.) For $s$ a complex number lying in the upper half-plane, let $\Omega_s$ be the stability datum given by the triple
$$
\Omega_{s}= (\beta,U,\rho_s(t)) := \left(3H-E_1-E_2, 1, 1 -st + \frac{s^2t^2}{2}\right).
$$ 
This stability datum satisfies all the requirements to define a polynomial central charge, except possibly the condition $\operatorname{Im}(s^2)>0$, which is merely a normalisation condition. After fixing a choice of Kähler metric $\omega \in \beta$, one checks that solving the $Z_{\Omega_s}$-critical equation on a line bundle $L$ on $X$ is equivalent to solving the equation 
\begin{equation}\label{eqn:example:CMA}
    \left(\chi_h - a\omega + \frac{b^2\beta^2-(c_1(L)-a\beta)^2}{2\beta \cdot (c_1(L)-a\beta)}\omega\right)^2 = b^2\left(1 + \left(\frac{b^2\beta^2-(c_1(L)-a\beta)^2}{2\beta\cdot (c_1(L)-a\beta)}\right)^2\right)\omega^2
\end{equation}

\noindent for a Hermitian metric $h$ on $L$. Here $a = \operatorname{Re}(s)$, $b = \operatorname{Im}(s)$, and $\chi_h$ is the curvature form $\frac{\i}{2\pi}F_h$ of the metric $h$. (We will furthermore require that $\beta\cdot (c_1(L)-a\beta) > 0$ for all line bundles $L$ under consideration.) In particular, we see immediately that the volume form hypothesis is always satisfied. In fact, the equation \eqref{eqn:example:CMA} is equivalent to a dHYM equation for the class $c_1(L)-a\beta$ with auxiliary Kähler metric $b\omega$.

Let $L_1 = \mathcal{O}_X(E_1)$ and $L_2 = \mathcal O_X(T)$. Then, one checks that for $a < \frac{1}{7}$, we have $\beta\cdot(c_1(L_i)-a\beta) > 0$ for $i=1,2$. Moreover, from \eqref{eqn:example:CMA} we can solve the $Z_{\Omega_s}$-critical equation on $L_i$ if and only if
$$
\tau(L_i,\Omega_s) := c_1(L_i)-a\beta + \frac{b^2\beta^2-(c_1(L_i)-a\beta)^2}{2\beta \cdot (c_1(L_i)-a\beta)}\beta
$$
is a Kähler class. (The class $\tau(L_i,\Omega_2)$ can never be the negative of a Kähler class, because for $b$ large and positive, it is clearly Kähler and $s\mapsto \tau(L_i,\Omega_s)$ is a continuous mapping into the disjoint union $\pm\mathcal P^+_X$.) A straightforward calculation (using the fact that $E_1$, $E_2$ and $T$ are the only curves of negative self-intersection) then shows that for $i = 1$ this happens precisely when 
$$
W_1(\Omega_s):= \int_{E_1}\tau(L_1,\Omega_s) = \frac{7b^2 + 8a}{1-7a} > 0
$$
and for $i=2$ if and only if
$$
W_2(\Omega_s):=\int_{T}\tau(L_2,\Omega_s) = \frac{7b^2 + 4(a+1)^2 -8}{5-11a} > 0.
$$
This gives us a two-dimensional local slice of the wall-chamber decomposition as shown in the figure below.
\begin{figure}[h!]
\centering
\begin{tikzpicture}[scale=2.5]
  \draw[<->] (-3.5, 0) -- (0.25, 0) node[right] {$\operatorname{Re}(s)$};
  \draw[<->] (0, -0.25) -- (0, 2) node[above] {$\operatorname{Im}(s)$};
  \draw[dashed] (1/7,0) node[below]{$\frac{1}{7}$} -- (1/7,2) ;
  \draw[ domain=36:180, smooth, variable=\a, line width = 0.3 mm] plot ({sqrt(2)*cos(\a)-1}, {sqrt(8/7)*sin(\a)});
  \draw[domain=0:2, smooth, variable=\b, line width = 0.3 mm]  plot ({-7*\b*\b/8}, {\b});
  \draw (-0.15,0.65) node[scale = 1]{(I)};
  \draw (-2.5,1) node[scale = 1]{(II)};
  \draw (-1,1.5) node[scale = 1]{(III)};
  \draw (-1.2,0.5) node[scale = 1]{(IV)};
  \draw (1/6,0.634) node[right]{$W_2(\Omega_s)=0$};
  \draw (-2.05,1.65) node[above]{$W_1(\Omega_s)=0$};
\end{tikzpicture}
\caption{{Wall-chamber decomposition of a slice of the `space of Z-critical equations' on $\mathrm{Bl}_{p_1,p_2}\mathbb{P}^2$, with chambers such that} (I): only $L_1$ is $Z_{\Omega_s}$-stable. (II): only $L_2$ is $Z_{\Omega_s}$-stable. (III): both $L_1$ and $L_2$ are $Z_{\Omega_s}$-stable. (IV): neither $L_1$ nor $L_2$ is $Z_{\Omega_s}$-stable.}\label{fig:example}
\end{figure}
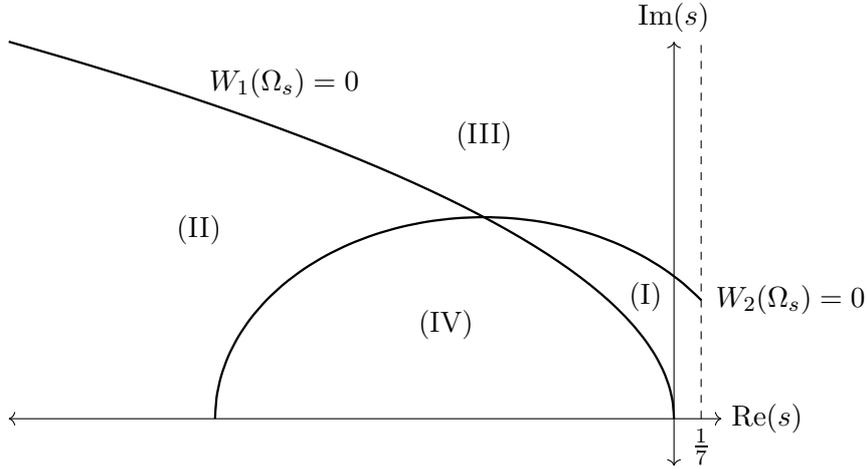

\begin{remark} Clearly one can produce many more examples and carry out a more general analysis of the above decomposition without much additional difficulty, in particular by making different choices for varying the stability datum $\Omega$. Other examples can be treated similarly as long as the boundary of the nef cone of $X$ is sufficiently well understood. We leave a more systematic study of similar and other examples to future work.
\end{remark}

\section{The special case of the J-equation: Finite set of optimal destabilizers in the small volume limit and applications} \label{Section J}

\noindent The J-equation is realized as the so-called \emph{small volume limit} of the deformed Hermitian Yang-Mills equation, and can thereby be seen as a simpler special case of the former. It is however interesting in its own right, as it is also closely connected to other geometric PDE of interest in K\"ahler geometry, such as the constant scalar curvature equation. Below we comment on this `special case' in some further detail, including optimal destabilizing curves for the J-equation, and test configurations, with relation to J-stability. 

In the following paragraphs we assume that $(X,\omega)$ is a compact K\"ahler surface with $\theta$ an auxiliary K\"ahler form on $X$, and denote the associated K\"ahler classes by $\alpha := [\omega]$ and $\beta := [\theta]$ respectively. These notations are sometimes used interchangeably, depending on whether we wish to emphasize the class or the underlying $(1,1)$-form.
Let moreover $c := c_{\theta,\omega} \in \mathbb{R}$ be the unique cohomological constant defined by the relation 
$$\int_X 2\theta \wedge \omega - c\omega^2 = 0.$$

For the sequel we moreover use the following terminology, which is a variant of Definition \ref{def sets of destabilizers} for the J-equation:  

\begin{definition} 
 The \emph{set of J-destabilizing curves for the pair $(\theta,\omega)$} is defined as
$$
\mathcal D_{\theta,\omega} := \mathcal D_{\tau_{\theta,\omega}} := \left\{E \subseteq X \; \mathrm{irreducible} \; \mathrm{curve} : 
\int_E \tau_{\theta,\omega} \leq 0 \right\}.
$$
where $\tau_{\theta,\omega}$ is the class $$ \tau_{\theta,\omega} := 2\frac{\int_X \omega\wedge\theta}{\int_X \omega^2}[\omega] - [\theta]. $$

\end{definition}

\begin{remark} Note that
\begin{enumerate}
    \item By Theorem \ref{Thm main Jequation} below the set $\mathcal D_{\theta,\omega}$ is finite.
    \item The form $\tau_{\theta,\omega}$ of course only depends on $[\theta]$ and $[\omega]$ in $H^{1,1}(X,\mathbb{R})$. Brackets are often removed in indices to alleviate notation.
\end{enumerate}
\end{remark}

\medskip

\subsection{Proof of the Main Theorem for the J-equation}

We then have the following main theorem for the J-equation. While the idea is the same as for Theorems \ref{Thm main dHYM intro} and \ref{Thm Main Theorem Z Critical Equation}, it requires a separate proof. 

\begin{theorem} \label{Thm main Jequation}
Let $X$ be a compact K\"ahler surface and fix any compact subset $K \subseteq \mathcal{C}_X \times \mathcal{C}_X$.
Then there exists a finite set $\mathcal{V}_K$ of curves of negative self-intersection $E_1, \dots, E_{\ell}$ on $X$ (depending only on $K$) such that the following are equivalent: 
\begin{enumerate}
    \item 
    For any Kähler metric $\theta \in \beta$, 
    there exists a Kähler metric $\omega \in \alpha$ solving the J-equation  \begin{equation} \label{Eq J}
    \mathrm{Tr}_{\omega}\theta = c.
    \end{equation}
    \item For all irreducible curves $E \subseteq X$ we have
    $$ 
    c \int_{E} \omega - \int_{E} \theta  > 0 
    $$
    for any smooth forms $\omega \in \alpha$ and $\theta \in \beta.$
    \item For all $E \in \mathcal{V}_K$ we have
    $$ 
    c \int_{E} \omega - \int_{E} \theta  > 0 
    $$ 
    for any smooth forms $\omega \in \alpha$ and $\theta \in \beta.$
\end{enumerate}
Moreover, if \eqref{Eq J} does not admit a solution and $
    c\int_{E} \omega - \int_{E} \theta \leq 0 
    $, then $E = E_i$ for some $i \in \left\{1,\dots,{\ell}\right\}$.
\end{theorem}

\begin{remark} Here
\begin{enumerate}
    \item $(1) \Leftrightarrow (2)$ is the solution of a conjecture of Lejmi-Sz\'ekelyhidi \cite{LejmiGabor}, recently proven in \cite{DatarPingali, Song}.
    \item The equivalence $(2) \Leftrightarrow (3)$ is more generally valid for smooth $(1,1)$-forms $\theta$ for which $\beta := [\theta]$ is in the open convex cone
$$ \mathcal P_X^+ = \left\{\gamma \in H^{1,1}(X,\mathbb R) \ | \ \gamma \cdot \alpha > 0, \gamma^2 > 0\right\},
$$
as explained in the proof below. 
\end{enumerate}
\end{remark}

\begin{proof}[Proof of Theorem \ref{Thm main Jequation}]
That $(2) \Rightarrow (3)$ is obvious. It remains to prove that $(3) \Rightarrow (2)$. To this end, let $\omega$ be a Kähler form on a compact Kähler surface $X$ with cohomology class $\alpha = [\omega] \in \mathcal C_X \subseteq H^{1,1}(X,\mathbb R)$, and let $\theta$ be any smooth real (1,1)-form on $X$ such that its class $\beta = [\theta]$ lies in $\mathcal P_X^+$.
We then observe that if $$ \tau := \frac{2 \int_X \theta\wedge\omega}{\int_X \omega^2} [\omega] - [\theta] $$ then we have $\tau\cdot [\omega] = [\theta]\cdot[\omega] > 0$ and $\tau^2 = [\theta]^2 > 0$, so $\tau \in \mathcal P_X^+$. Suppose moreover that equation \eqref{Eq J} does not admit a solution with respect to $(\theta,\omega)$. By \cite[Theorem 2]{Chen2000} this is equivalent to $\tau$ not being a Kähler class. By Lemma \ref{Lemma positive inside big} we however have $\mathcal B_X \supseteq  \mathcal P_X^+$, and hence $\tau \in \mathcal{B}_X \setminus \mathcal{C}_X$. Then, by Proposition \ref{Prop meta}, there exist at most finitely many curves $E \subseteq X$ such that $$ \int_E \tau \leq 0.$$ 

\noindent Finally, the image of the compact set $K$ under the map 
$$
\mathcal{P}_X^+ \times \mathcal{C}_X \ni ([\theta],[\omega]) \mapsto \tau_{\theta,\omega} := c_{\theta,\omega}[\omega] - [\theta] \in \mathcal{P}_X^+
$$
is clearly compact, so by Lemma \ref{Lemma convex hull} it is contained in the convex hull of a finite number of classes $\alpha_1, \dots, \alpha_k$. By Lemma \ref{Lem cardinality bound cvx hull} this implies that $|\mathcal{V}_K| \leq k\rho(X)$. Finally, the fact that any J-destabilizing curve must already appear in the list $\left\{E_1, \dots, E_{\ell}\right\}$ is due to Proposition \ref{Prop meta}. Since clearly $\mathcal{C}_X \subseteq \mathcal{P}_X^+$ this finishes the proof. 
\end{proof} 

\noindent As a direct consequence of the above proof, we highlight the following more precise cardinality bound: 

\begin{corollary}
In the notation of the proof of Theorem \ref{Thm main Jequation}, suppose that the image of $K$ under the map 
$$
\mathcal{P}_X^+ \times \mathcal{C}_X \ni ([\theta],[\omega]) \mapsto \tau_{\theta,\omega} := c_{\theta,\omega}[\omega] - [\theta] \in \mathcal{P}_X^+
$$
is contained in the convex hull of $k$ pseudoeffective classes. Then 
$$
\ell = |\mathcal{V}_K| \leq k\rho(X)
$$
\end{corollary}

\begin{remark} \emph{(Perturbation of $[\omega]$ and blowups)} The above shows a certain robustness of the set of destabilizing curves when varying the initial data $([\theta],[\omega])$. Indeed, if we consider 
$
\mathcal S(\theta,\omega) := \mathrm{Neg}(\tau_{\theta,\omega}),
$
which is a set of candidate destabilizing curves (as indeed $\mathcal{S}(\theta,\omega) \subseteq \mathcal{D}_{\tau_{\theta,\omega}}$ in the sense of Definition \ref{def sets of destabilizers}), then 
we record the following elementary consequences of analogous properties for the Zariski decomposition (see Theorem \ref{Thm Zariski decomposition}):
Suppose that $([\theta],[\omega]) \in \mathcal{P}_X^+ \times \mathcal{C}_X$. Then for any $\epsilon > 0$ we have
$$
\mathcal{S}(\theta + \epsilon\omega,\omega) \subseteq \mathcal{S}(\theta,\omega).
$$
Moreover
$$
\mathcal{S}(\theta,r\omega) = \mathcal{S}(\theta,\omega)
$$
for any $r > 0$, and if $\pi:Y_p \rightarrow X$ is the blowup of $X$ at the point $p \in X$, with exceptional divisor $E$, then $$
\mathcal{S}(\pi^*\theta,\pi^*\omega) = \widetilde{\mathcal{S}(\theta,\omega)} \cup \left\{E\right\}.
$$
where $\widetilde{\mathcal{S}}$ denotes the strict transform of $\mathcal{S}$.
The first part follows immediately from Theorem \ref{Thm Zariski decomposition}, as well as the observation that $\tau_{\theta + \epsilon\omega,\omega} = \tau_{\theta,\omega} + \epsilon$ and $\tau_{\theta,r\omega} = \tau_{\theta,\omega}$. The second part is Lemma \ref{Lem Zariski Birational}.
This can also be seen from the point of view of Siu decompositions (see discussion in Remark \ref{Rmk Discussion Meta prop}), by noting that the Siu decomposition can absorb small perturbations of $\alpha$ (and therefore the same set $\mathcal{S}$ can be taken across a small enough neighbourhood). However, realizing the behaviour across birational transformations is a convenient feature of the Zariski decomposition.

In particular, this gives a good control of the set of J-destabilizing curves as we vary the underlying classes $([\theta],[\omega]) \in K \subseteq \mathcal{C}_X \times \mathcal{C}_X$, but also as we modify our surface $X$ via certain birational transformations. Applications of this latter point would be interesting to explore further in the future.
\end{remark}

\subsection{Optimally destabilizing curves for the J-equation and deformed Hermitian Yang-Mills equation} \label{Subsection optimal dest Jstab}

Apart from studying the set of all destabilizing curves, it is furthermore natural to ask if it is possible to determine when \emph{optimally} destabilizing curves exist (in the sense of realizing a suitable infimum), and if so how they may be characterized. As in \cite{SD4}, consider 
\begin{equation} \label{Eq deltaNM}
\Delta_{\mathrm{NM}}(\theta,\omega) := \inf_{E \subseteq X} \left(c_{\theta,\omega} - \frac{\int_E \theta}{\int_E \omega}\right),
\end{equation}
(so that $\Delta_{\mathrm{NM}}(\theta,\omega) > 0$ if and only if the J-equation \eqref{Eq J} is solvable with respect to $(\theta,\omega)$), and write
$$
\mathcal{D}^{\mathrm{J}}_{\theta,\omega} := \left\{ E \subseteq X \textrm{ irreducible curve } | \ \Delta_{\mathrm{NM}}(\theta,\omega) = c_{\theta,\omega} - \frac{\int_E \theta}{\int_E \omega} \leq 0 \right\}
$$
for the (possibly empty) set of optimally J-destabilizing curves, i.e. curves $E \subseteq X$ that destabilize and precisely realize the infimum in \eqref{Eq deltaNM}. As a first elementary observation we note that the above set is invariant under changing $\theta$ linearly along paths $\theta_t := (1-t)\omega + t\theta$, $t \geq 0$, as long as $\Delta_{\mathrm{NM}}(\theta_t,\omega) \leq 0$: 

\begin{lemma}
Along linear perturbations $\theta_t := (1-t)\omega + t\theta$, $t \geq 0$, the set $\mathcal{D}^{\mathrm{J}}_{\theta_t,\omega}$ is independent of $t \in [0,+\infty)$. It is non-empty precisely when $\Delta_{\mathrm{NM}}(\theta,\omega) \leq 0$.
\end{lemma} 

\begin{proof}
The independence in $t$ is an immediate consequence of the linearity of the function $[0,+\infty) \ni t \mapsto \Delta^{\mathrm{pp}}_{\theta_t}(\omega)$, see \cite[Lemma 15]{SD4}, and the straightforward observation that every curve optimally destabilizes at $t = 0$. The non-emptyness is an application of Theorem \ref{Thm main Jequation}, as the infimum over a finite set is always realized.
\end{proof}

\noindent We can further describe the set of J-optimal destabilizers as follows. The argument is purely cohomological, but to alleviate notation we omit brackets in indices: Let $\theta$ and $\omega$ be K\"ahler forms on $X$ and define 
$$
\mathbf{a} := \mathbf{a}([\theta],[\omega]) \in H^{1,1}(X,\mathbb{R})
$$ 
as the unique nef class such that 
\begin{enumerate}
    \item The classes $[\omega_t] := (1-t)\mathbf{a} + t[\theta]$ are K\"ahler for all $t \in [0,1)$, 
     \item $\mathbf{a} = [\omega_1]$ is not K\"ahler, and 
    \item $[\omega] = [\omega_t]$ for some $t \in (0,1)$.
\end{enumerate}

\noindent Geometrically, $\mathbf{a}$ is thus the first non-K\"ahler, but big, $(1,1)$-class following a given direction in the K\"ahler cone (in case $[\theta] = K_X$ this is often referred to as the \emph{nef threshold}). 
As this class $\mathbf{a}$ is big but not K\"ahler, the Nakai-Moishezon criterion implies that there must exist curves $C \subseteq X$ such that $\mathbf{a}.[C] = 0$. There is a finite number of such curves, and it turns out that they play a central role for both the dHYM equation and the J-equation, as described by the following result: 

\begin{proposition} \label{Prop optimally Jneg curves} \emph{(Optimally J-destabilizing curves)}
Let $([\theta],[\omega])$ be K\"ahler classes on $X$. 
In the above notation, if $\mathbf{a} := \mathbf{a}([\theta],[\omega]) \in H^{1,1}(X,\mathbb{R})$ is in $\mathcal{P}_X^+$ and  $\Delta_{\mathrm{NM}}(\theta,\omega) \leq 0$, then the set of optimally J-destabilizing curves is given by
$$
\mathcal{D}^{\mathrm{J}}_{\theta,\omega} = \left\{E \subseteq X \; \vert \; \int_E \mathbf{a} = 0 \right\},
 $$
and the cardinality 
$|\mathcal{D}^{\mathrm{J}}_{\theta,\omega}|$ 
does not exceed $\rho(X)$.
\end{proposition}
\begin{proof}
The proof is purely cohomological. First note that along $(\theta,\omega_t)$ we have
$$
\Delta_{\mathrm{NM}}(\theta,\omega_t) = c_{\theta,\omega_t} - t^{-1},
$$
as follows from the computation in \cite[Lemma 16]{SD4}. 
On the other hand, it is a direct consequence of Theorem \ref{Thm main Jequation} that the above infimum can be taken over a finite set, and is thus realized by a curve $E \subseteq X$, i.e. 
$$
\Delta_{\mathrm{NM}}(\theta,\omega_t) = c_{\theta,\omega_t} - \frac{\int_E \theta}{\int_E \omega_t}.
$$
A necessary and sufficient condition for being optimally destabilizing is therefore that
$$
\int_E [\omega_t] = \int_E (1-t)\mathbf{a} + t[\theta] = t\int_E [\theta],
$$
which happens precisely if $\int_E \mathbf{a} = 0$.
The cardinality bound follows from Proposition \ref{Prop meta}, since 
$$
\mathcal{D}^{\mathrm{J}}_{\theta,\omega} = \left\{E \subseteq X \; \vert \; \mathbf{a}.[E] = 0 \right\} \subseteq \mathcal{S}(\theta,\omega).
$$
\end{proof}

\noindent In other words the curves that `most destabilize' the numerical condition relevant to the J-equation are precisely the curves that obstruct the boundary class $\mathbf{a} \in \partial \mathcal{C}_X$ from being K\"ahler, according to the Nakai-Moishezon criterion. This turns out to be a general truth coming from the dHYM equation (see below), and the above result will then illustrate that the set of optimally dHYM-destabiizing curves is preserved in the small volume limit, as $[\omega] \mapsto [k\omega]$, with $k \rightarrow +\infty$, in a suitable sense. 

To this end, we ask the analogous question for the dHYM equation, under the natural assumption that $[\omega]^2 - [\theta]^2 > 0$ (equivalent to the supercritical phase condition, see Remark \ref{Rmk dHYM phase}). Then let 
$$
\Delta_{\mathrm{dHYM}}(\theta,\omega) := \left\{ E \subseteq X \textrm{ irreducible curve } | \ \Delta_{\mathrm{dHYM}}(\theta,\omega) = (\widetilde{c_{\theta,\omega}})^{-1} - \frac{\int_E \theta}{\int_E \omega} \leq 0 \right\},
$$
where 
$$
\widetilde{c_{\theta,\omega}} := \frac{[\omega]^2 - [\theta^2]}{2[\theta].[\omega]}
$$
is a positive number under the above hypotheses. 
The same argument then applies also to the dHYM equation, and yields the following:

\begin{proposition} \label{Prop optimally dHYMneg curves} \emph{(Optimally dHYM-destabilizing curves)}
In the setup of Proposition \ref{Prop optimally Jneg curves}, suppose that $\mathbf{a}([\theta],[\omega]) \in \mathcal{P}_X^+$, $[\theta]^2 < [\omega]^2$, and  $\Delta_{\mathrm{dHYM}}(\theta,\omega) \leq 0$. Then the set of all optimally dHYM-destabilizing curves is  given by
$$
\mathcal{D}^{\mathrm{dHYM}}_{\theta,\omega} = \left\{E \subseteq X \; \vert \; \int_E \mathbf{a} = 0 \right\}. 
 $$
\end{proposition}

\begin{proof}
The proof is a repetition of that for Proposition \ref{Prop optimally Jneg curves}, with $c_{\theta,\omega}$ replaced by $(\widetilde{c_{\theta,\omega}})^{-1}$.
\end{proof}

\noindent In particular, whenever non-empty, the sets of J/dHYM destabilizing curves in fact \emph{coincide}.  

\begin{proof}[Proof of Theorem \ref{Thm opt J vs opt dHYM intro}]
The statement of the theorem is immediately deduced by combining Propositions \ref{Prop optimally Jneg curves} and \ref{Prop optimally dHYMneg curves} above.
\end{proof}

\noindent 
The interesting question of investigating the analogous question for the Z-critical equation with general central charge is left for future work. 

\medskip

\subsection{Remarks on optimally destabilizing test configurations for J-stability}

\noindent Motivated by the well-known circle of ideas relating solvability of geometric PDE to analytic and algebro-geometric stability, we relate the above discussion to J-stability and coercivitiy of the J-functional (see \cite{LejmiGabor, CollinsGabor, GaoChen}, Section \ref{Section prelim}). The first observation is that if we consider the following natural numerical quantities
$$
\Delta^{\mathrm{alg}}(\theta,\omega) := \sup \left\{\delta \in \mathbb{R}\; : \; \mathrm{E}_{\theta}^{\mathrm{NA}}(\Phi) \geq \delta ||\Phi|| \; \forall \Phi \in \mathcal{H}(X,[\omega])\right\}.
$$
$$
\Delta(\theta,\omega) := \sup \left\{\delta \in \mathbb{R} \; : \; \exists C > 0 : \; \mathrm{E}_{\theta,\omega}(\varphi) \geq \delta \mathrm{E}_{\omega,\omega}(\varphi) - C \; \forall \varphi \in \mathcal{H}(X,\omega)\right\}.
$$
then these infima in fact all coincide whenever one of them is destabilized: 

\begin{corollary} \label{Cor Jstab}
Suppose that $\theta$ and $\omega$ are K\"ahler forms on $X$. Assume moreover that $\tau:=\tau_{\theta,\omega}$ is not a Kähler class. Then there is a non-empty finite set of curves $E_1, \dots, E_{\ell}$ on $X$, of cardinality not exceeding $\rho(X)$, such that 
$$
 \Delta^{\mathrm{alg}}(\theta,\omega) = \Delta(\theta,\omega) = \Delta_{\mathrm{NM}}(\theta,\omega) = \inf_{i = 1,\dots, \ell} \frac{\int_{E_i} \tau}{\int_{E_i} \alpha} 
 $$ 
\end{corollary}

\begin{proof}
From Theorem \ref{Thm main Jequation} we conclude that there exist at most finitely many curves $E_i\subseteq X$ such that $$ \Delta_{\mathrm{NM}}(\theta,\omega) = \inf_{E_i \subseteq X} \left(c_{\theta,\omega} -  \frac{\int_{E_i} \tau}{\int_{E_i} \alpha}\right),$$
and the infimum over a finite set must be realized, in this case by one of the curves $E_1, \dots, E_{\ell}$, where $\ell \leq \rho(X)$ by Proposition \ref{Thm main Jequation}.
Moreover, the condition on $\tau$ not being Kähler translates precisely into $\Delta_{\mathrm{NM}}(\theta,\omega)\leq 0 < \sup\left\{\lambda \in \mathbb R \ | \ [\theta] - \lambda[\omega] \geq 0 \right\}$. Thus, by \cite[Theorem 1 and 8]{SD4}, we obtain also that $$ \Delta(\theta,\omega) = \Delta^{\mathrm{alg}}(\theta,\omega) = \Delta_{\mathrm{NM}}(\theta,\omega),
$$
concluding the proof. 
\end{proof}

\noindent We underline however that the above result does \emph{not} yet say anything about whether or not the threshold $\Delta^{\mathrm{alg}}(\theta,\omega)$ is realized by a test configuration, i.e. whether there exists $\Phi \in \mathcal{H}(X,[\omega])$ realizing the infimum
$$
\Delta^{\mathrm{alg}}(\theta,\omega) = \frac{\mathrm{E}_{\theta}^{\mathrm{NA}}(\Phi)}{||\Phi||}.
$$

\noindent The answer to this question turns out to be  \emph{negative}. As a preliminary lemma, we show this for the special class of test configurations $\Phi_{C,\kappa}$ coming from the deformation to the normal cone construction (see Section \ref{Section prelim} for the notational conventions used):

\begin{lemma} \label{Lemma no destabilizing tc}
Suppose $X$ is a smooth projective surface and that $([\theta],[\omega]) \in \mathcal{C}_X \times \mathcal{C}_X$ such that $\Delta^{\mathrm{alg}}(\theta,\omega) \leq 0$. Let moreover $C \subseteq X$ be any irreducible curve on $X$. Then we always have a strict inequality
\begin{equation} \label{Eq inf alg threshold}
\Delta^{\mathrm{alg}}(\theta,\omega) <  \frac{\mathrm{E}_{\theta}^{\mathrm{NA}}(\Phi_{C,\kappa})}{||\Phi_{C,\kappa}||}
\end{equation}
for every $\kappa \in (0,\bar{\kappa}_C)$.
\end{lemma}

\begin{proof}
Introduce the shorthand notation
$$
A_C := c_{\theta,\omega}\int_C \omega - \int_C \theta, \; \; B_C := -\frac{2[\omega].[\theta]}{3[\omega]^2}C^2.
$$
As a consequence of Lemma \ref{Lemma polynomial expansion}, the condition for $\mathrm{J}^{\theta}$-semistability is that
$
A_C + B_C\kappa \geq 0
$
is satisfied for all curves $C \subseteq X$ and for all $\kappa \in (0,\bar{\kappa}_C)$ smaller than the Seshadri constant 
$$
\bar{\kappa}_C := \sup \left\{ \epsilon > 0: \pi^*c_1(L) - \epsilon[E] > 0\right\}.
$$
Likewise, the condition for $J^{\theta}$-stability is the same except that the inequality above is required to be strict, i.e. $A_C + B_C\kappa > 0$ for all curves $C \subseteq X$ and for all $\kappa \in (0,\bar{\kappa}_C)$. 

Now assume that $\Delta^{\mathrm{alg}}(\theta,\omega) = 0$, i.e. $(X,[\omega])$ is $\mathrm{J}^{\theta}$-semistable but not uniformly $\mathrm{J}^{\theta}$-stable. We then claim that the infimum in \eqref{Eq inf alg threshold} is \emph{not} realized by any relatively K\"ahler test configuration $\Phi_{C,\kappa}$. To see this, first of all note that semistability implies $A_C \geq 0$. We then treat the cases $C^2 < 0$ and $C^2 \geq 0$ separately: in the case $C$ is a negative curve, the leading term $B_C > 0$, so that $A_C + B_C\kappa > 0$ for any positive $\kappa > 0$. If we instead assume that $C^2 \geq 0$, then $B_C \leq 0$ and Theorem \ref{Thm main Jequation} implies that $A_C > 0$ with a \emph{strict} inequality (since $A_C \leq 0$ would force $C$ to be one of the finitely many curves of negative self-intersection produced by the theorem, contradicting $C^2 \geq 0$).
The case $B_C = 0$ thus clearly does not lead to a destabilizing relatively K\"ahler test configuration $\Phi_{C,\kappa}$, and it remains to treat the case when $A_C, B_C > 0$. If we then suppose for contradiction that $A_C + B_C\kappa = 0$ for some $\kappa \in (0,\bar{\kappa}_C)$, then the openness of the interval of admissible $\kappa$ implies that $\kappa + \epsilon \in (0,\bar{\kappa}_C)$ for  $\epsilon > 0$ small enough, and at the same time $A + B(\kappa + \epsilon) < 0$, contradicting the semistability assumption. 
To summarize, we must have
$$
\mathrm{E}_{\theta}^{\mathrm{NA}}(\Phi_{C,\kappa}) = A_C + B_C\kappa > 0
$$
for all $C \subseteq X$ and all $\kappa \in (0,\bar{\kappa}_C)$. In particular, the infimum in \eqref{Eq inf alg threshold} is not realized and $(X,[\omega])$ thus satisfies the condition for $\mathrm{J}^{\mathrm{\theta}}$-stability across all relatively K\"ahler test configurations of the form $\Phi_{C,\kappa}$ for $C \subseteq X$ irreducible curves. 

Finally, suppose more generally that $\Delta_{\theta}^{\mathrm{alg}}(\omega) = -R$ for some $R \geq 0$. Then $[\theta_R] := [\theta] + R[\omega]$ is a K\"ahler class, and moreover
\begin{equation} \label{Eq linear NA relation}
\frac{\mathrm{E}_{\theta_R}^{\mathrm{NA}}(\Phi_{C,\kappa})}{||\Phi_{C,\kappa}||} = \frac{\mathrm{E}_{\theta}^{\mathrm{NA}}(\Phi_{C,\kappa})}{||\Phi_{C,\kappa}||} + R
\end{equation}
for every test configuration $\Phi_{C,\kappa}$. 
Moreover
$
\Delta^{\mathrm{alg}}(\theta_R, \omega) = \Delta^{\mathrm{alg}}(\theta, \omega) + R = 0, 
$
reducing the argument to the previous case. Indeed, by the above argument no test configuration realizes the infimum defining $\Delta^{\mathrm{alg}}(\theta_R, \omega)$, and by \eqref{Eq linear NA relation} the same conclusion holds also for the infimum defining $\Delta^{\mathrm{alg}}(\theta,\omega)$, concluding the proof.
\end{proof}

\noindent In other words, since the inequality in the statement of Lemma \ref{Lemma no destabilizing tc} is strict, these test configurations never optimally destabilize (although in the limit $\kappa \rightarrow 0$ the right hand side in Lemma \ref{Lemma no destabilizing tc} converges to the left hand side, for suitably chosen curves $C \subseteq X$). To summarize, \emph{if one restricts to normal and relatively K\"ahler test configurations of the special form $\Phi_{C,\kappa}$ coming from deformation to the normal cone}, then (restricted) J-semistability would be equivalent to (restricted) J-stability, and these are both closed conditions in the K\"ahler cone (the fact that J-semistability is a closed condition is well-known, see e.g. \cite{SD4}). Moreover, it is known that uniform J-stability is equivalent to existence of solutions to the J-equation (by \cite{Chen2000, GaoChen}), and these conditions are in general not equivalent to J-stability.

Using a recent result of Hattori \cite{Hattori} we now prove that most of the above picture remains true over the larger set of \emph{all normal and relatively K\"ahler} test configurations:

\begin{proposition} \label{Prop Jstab vs UJstab}
Suppose that $H$ and $L$ are ample $\mathbb{Q}$-line bundles on a smooth projective surface $X$. Then $(X,L)$ is $\mathrm{J}^{\mathrm{H}}$-semistable if and only if it is $\mathrm{J}^{\mathrm{H}}$-stable. Moreover, if $[\theta]$ and $[\omega]$ are rational $(1,1)$-classes on $X$ such that there is no solution to \eqref{Eq J}, then there does not exist any optimally destabilizing relatively K\"ahler test configuration for $([\theta],[\omega])$, i.e.
$$
\Delta^{\mathrm{alg}}(\theta,\omega) <  \frac{\mathrm{E}_{\theta}^{\mathrm{NA}}(\Phi)}{||\Phi||}
$$
for all normal and relatively K\"ahler test configurations $\Phi \in \mathcal{H}^{\mathrm{NA}}(X,[\omega])$. 
\end{proposition}

\begin{proof}
Suppose that $(X,L)$ is  $\mathrm{J}^{\mathrm{H}}$-semistable. By Lemma \ref{Lemma polynomial expansion} (originally due to \cite[Proposition 13]{LejmiGabor} based on computations in \cite{RossThomas}) it is elementary to see that the $\mathbb{Q}$-line bundle
$$
2\frac{H.L}{L^2}L - H 
$$
is nef (if not, we could find a test configuration $\Phi_{C,\kappa}$ with $\kappa$ small enough such that $\mathrm{J}^{\mathrm{NA}}(\Phi_{C,\kappa}) < 0$, thus contradicting semistability).  
The same is true replacing $(H,L)$ by $(rH,rL)$ where $r > 0$ is chosen such that $rc_1(H),rc_1(L) \in H^{1,1}(X,\mathbb{Z})$. One may then apply \cite[Proposition 7.4]{Hattori} which in turn implies that $(X,rL)$ is $\mathrm{J}^{\mathrm{rH}}$-stable (which is equivalent to $(X,L)$ being $\mathrm{J}^{\mathrm{H}}$-stable), proving the first part.
Conversely, stability implies semistability by definition.

For the final part of the statement, suppose that $\Delta^{\mathrm{alg}}(\theta, \omega) = -R$ for some $R \geq 0$. Then $[\theta_R] := [\theta] + R[\omega]$ is a K\"ahler class, and moreover
$
\Delta^{\mathrm{alg}}(\theta_R, \omega) = \Delta^{\mathrm{alg}}(\theta, \omega) + R = 0,
$
so that $(X,[\omega])$ is $\mathrm{J}^{\theta_R}$-semistable. Moreover, it follows from \cite[Theorem 1]{SD4} and Theorem \ref{Thm main Jequation} that 
$$
R = 2\frac{[\theta].[\omega]}{[\omega]^2} - \frac{\int_C \theta}{\int_C \omega} 
$$
for some irreducible curve $C \subseteq X$. Since $R \in \mathbb{Q}$ the $(1,1)$-class $[\theta_R]$ is rational, and we may apply the first part of the proof to conclude that $(X,[\omega])$ is also $\mathrm{J}^{\theta_R}$-stable. By definition of stability we thus have
$$
\frac{\mathrm{E}_{\theta_R}^{\mathrm{NA}}(\Phi)}{||\Phi||} > 0 = \Delta^{\mathrm{alg}}(\theta_R,\omega)
$$
for all $\Phi \in \mathcal{H}^{\mathrm{NA}}(X,[\omega])$. 
Moreover:
$$
\frac{\mathrm{E}_{\theta}^{\mathrm{NA}}(\Phi)}{||\Phi||} = \frac{\mathrm{E}_{\theta_R}^{\mathrm{NA}}(\Phi)}{||\Phi||} - R > \Delta^{\mathrm{alg}}(\theta_R, \omega) - R = \Delta^{\mathrm{alg}}(\theta, \omega),
$$
concluding the proof. 
\end{proof}

\noindent Emphasizing the relationship between test configurations for $(X,[\omega])$ with $[\omega] = c_1(L)$ for some ample line bundle on $X$, and test configurations for $(X,L)$ (see \cite{Zakthesis}), and that optimally destabilizing test configurations can by definition only exist if the J-equation is not solvable, the above result in particular yields the following:
\begin{corollary}
Let $X$ be a smooth projective surface with $L$ and $H$ two $\mathbb{Q}$-line bundles on $X$. Then no optimally destabilizing normal and relatively K\"ahler test configuration exists for $(X,L,H)$, in the sense of \eqref{Eq opt dest tc}.
\end{corollary}

\noindent Building further on \cite{Hattori} combined with \cite{SD4}, we deduce that examples of 
J-stable but not uniformly J-stable K\"ahler classes
always exist on compact K\"ahler surfaces that admit at least one negative curve: 

\begin{theorem} \label{Cor many Hattori examples}
Suppose that $X$ is a smooth projective K\"ahler surface admitting at least one negative curve. Then there exist  ample $\mathbb{Q}$-line bundles $L$ and $H$ on $X$ such that $(X,L)$ is $J^H$-stable but not uniformly $J^H$-stable.
\end{theorem}

\begin{remark}
If by contrast $X$ admits no negative curves, then by \cite{SD4} the $J^H$-equation can always be solved for $(X,L)$ for every pair of ample line bundles $(L,H)$ on $X$, and by \cite{LejmiGabor, CollinsGabor, DervanRoss}, $(X,L)$ is then always both $J^H$-stable and uniformly $J^H$-stable, for all such pairs. Together with Theorem \ref{Cor many Hattori examples} above, this gives a satisfactory answer in all possible cases when $X$ is a projective surfaces. 
\end{remark}

\begin{proof}[Proof of Theorem \ref{Cor many Hattori examples}]
First note that if $X$ is projective and admits at least one negative curve, then one can always find a rational nef and big, but not K\"ahler, class $[\omega'] \in H^{1,1}(X,\mathbb{Q})$. Indeed, since $X$ is projective, pick an ample divisor $A$, and since $X$ admits a curve of negative self-intersection, pick a curve $C$ such that $C^2 <0$. One can moreover check that 
$$
\sup\left\{t>0 \; | \; A + tC \textrm{ is nef.}\right\} = -\frac{A\cdot C}{C^2} \in \mathbb Q.
$$
Then the $\mathbb{Q}$-divisor 
$
B := A-\frac{A\cdot C}{C^2}C
$
is pseudoeffective and nef, but not ample. Because  
$$
B^2 = A^2 -\frac{(A\cdot C)^2}{C^2} > 0,
$$
the class $B \in \mathcal{B}_X$, i.e. $B$ is also big. 

We therefore pick such a nef and big, but not K\"ahler, $[\omega'] \in H^{1,1}(X,\mathbb{Q})$ and without loss of generality normalize $[\omega']$ and $[\theta]$ such that $[\omega']^2 = [\theta]^2$. Let $\mathrm{Js}^{\theta}$ (resp. $\mathrm{UJs}^{\theta}$) denote the set of K\"ahler class $[\omega]$ such that $(X,[\omega])$ is $\mathrm{J}^{\theta}$-stable (resp. uniformly $\mathrm{J}^{\theta}$-stable). 
Now if $X$ admits a negative curve, then there exists at least one pair $([\theta],[\omega]) \in \mathcal{C}_X \times \mathcal{C}_X$ of K\"ahler classes on $X$ such that $(X,[\omega])$ is not $\mathrm{J}^{\theta}$-semistable (see \cite[Theorem 6]{SD4}). By continuity (of $\Delta_{\mathrm{NM}}(\theta,\omega)$, see \cite[Theorem 1]{SD4}) one can then argue that there exists a $\mathrm{J}^{\theta}$-semistable class which is not uniformly $\mathrm{J}^{\theta}$-stable, and it remains to see that under the above hypotheses there exists a \emph{rational} such class $[\omega]$. To see this, note that since $[\omega']$ and $[\theta]$ are both $\mathbb{Q}$-classes by construction, then so is
$$
[\omega_t] := (1-t)[\omega] + t[\theta],
$$
for every $t \in \mathbb{Q}$. By a direct computation identical to that in \cite[Proposition 16 and Equation (14)]{SD4} we moreover see that 
$[\omega_t] \in \mathrm{Js}^{\theta} \setminus \mathrm{UJs}^{\theta} \neq \emptyset
$
if and only if
$$
2\frac{[\theta].[\omega_t]}{[\omega_t]^2}[\omega_t] - t^{-1} = 0.
$$
which happens precisely for $t = 1/2$ (by the normalization assumed at the beginnning of the proof; this is moreover of course equivalent to showing that $\tau_{\theta,\omega_t}$ is nef but not K\"ahler precisely when $t = 1/2$). Since $1/2 \in \mathbb{Q}$, $[\omega_{1/2}] \in H^{1,1}(X,\mathbb{Q})$, and Proposition \ref{Prop Jstab vs UJstab} shows that $(X,[\omega_{1/2}])$ is $J^{\theta}$-stable but not uniformly $J^{\theta}$-stable, concluding the proof.
\end{proof}

\begin{remark}
This sheds additional light on a recent example due to Hattori \cite{Hattori}, of a polarized smooth surface $(X,L)$ and an ample line bundle $H$ such that $(X,L)$ is $J^H$-stable but not uniformly. Indeed, the above result shows that such examples exist in abundance, namely on any compact K\"ahler surface with at least one negative curve, for well chosen $\mathbb{Q}$-line bundles.
\end{remark}

\noindent 
Finally, it is therefore natural to ask about the structure of the set of J-stable but not uniformly J-stable K\"ahler classes, which is non-empty by the above. 

\subsection{Real algebraic boundary}

\noindent As an application of Theorem \ref{Thm main Jequation}, we deduce that the (non-empty) set of J-stable but not uniformly J-stable K\"ahler classes is cut out by real algebraic sets. To see this, let $K \subseteq \mathcal{C}_X \times \mathcal{C}_X$ be any given compact set, and note that when the set
$$
\mathrm{Js} = \left\{([\theta],[\omega]) \in K: (X,[\omega]) \; \mathrm{J^{\theta}-stable} \right\}
 $$
is restricted to $\mathbb{Q}$-classes, it coincides with the set $\mathrm{Jss}$ of $([\theta],[\omega]) \in K$ such that $(X,[\omega])$ is $\mathrm{J}^{\theta}$-semistable. In turn, $\mathrm{Jss}$ is cut out by finitely many conditions of the form
$$
 \int_{C_i}(2([\theta].[\omega])[\omega] - [\omega]^2[\theta]) \geq 0,
$$
following Theorem \ref{Thm main Jequation}, where $k \in \mathbb{N}$ is such that the image of $K$ under the map $\Psi$ given by $\mathcal{C}_X \times \mathcal{C}_X \ni ([\theta],[\omega]) \mapsto \tau_{\theta,\omega} \in H^{1,1}(X,\mathbb{R})$ is contained in the convex hull of $k$ pseudoeffective classes on $X$.
Along these lines, we conclude the following: 

\begin{proposition}
Let $K \subseteq \mathcal{C}_X \times \mathcal{C}_X$ be any compact subset whose image under $\Psi$ is contained in the convex hull of $k$ pseudoeffective classes, and consider the $(1,1)$-class
$$
\tau_{\theta,\omega} := 2([\theta].[\omega])[\omega] - [\omega]^2[\theta] \in H^{1,1}(X,\mathbb{R}). 
$$
Then the set
$$
\left(\mathrm{Jss} \setminus \mathrm{UJs}\right)_{\vert K} 
= \left\{([\theta],[\omega]) \in K: (X,[\omega]) \; \mathrm{is} \; \mathrm{J^{\theta}-semistable} \; \mathrm{but} \; \mathrm{not} \; \mathrm{uniformly} \; \mathrm{J^{\theta}-stable} \right\}
$$
is real algebraic. More explicitly, there is a finite set of curves $\left\{C_1, \dots, C_{\ell} \right\}$ of cardinality $\ell \leq k\rho(X)$, 
such that 
$$
\left(\mathrm{Jss} \setminus \mathrm{UJs}\right)_{\vert K} = \left\{([\theta],[\omega]) \in K: \prod_{i = 1}^{\ell} \int_{C_i} \tau_{\theta,\omega} = 0 \right\}.
$$
\end{proposition}

\begin{proof}
First note that uniform $J^{\theta}$-stability is equivalent to $\int_C \tau_{\theta,\omega} > 0$, and $J^{\theta}$-semistability is equivalent to $\int_C \tau_{\theta,\omega} \geq 0$, for all curves $C \subseteq X$. Indeed, the former equivalence is due to \cite{GaoChen}, and to prove the second we first note that $J^{\theta}$-semistability implies that $\int_C \tau_{\theta,\omega} \geq 0$. If $X$ is projective this is a direct consequence of Lemma \ref{Lemma polynomial expansion}. In general, it follows from the standard formula $$\mathrm{E}^{\mathrm{NA}}_{\theta + \epsilon\omega,\omega} = \mathrm{E}^{\mathrm{NA}}_{\theta,\omega} + \epsilon||.||,$$ that $J^{\theta}$-semistability implies uniform $J^{\theta + \epsilon\omega}$-stability for any $\epsilon > 0$. This in turn implies that $\tau_{\theta + \epsilon\omega,\omega} = \tau_{\theta,\omega} + \epsilon \omega$ is K\"ahler for all $\epsilon >0$, by \cite{GaoChen}. Hence $\tau_{\theta,\omega}$ must be nef.

Moreover, it is immediate from the main Theorem \ref{Thm main Jequation} to see that both sets 
\begin{equation} \label{Equation 1}
\mathrm{Jss}_{\vert K} = \left\{([\theta],[\omega]) \in K : (X,[\omega])\;\mathrm{is}\;J^{\theta}-\mathrm{semistable} \right\}
\end{equation}
and
\begin{equation} \label{Equation 2}
\left\{([\theta],[\omega]) \in K : \int_C \tau_{\theta,\omega} \geq 0 \right\}.
\end{equation}
are realized as the \emph{smallest} closed sets containing 
$$
\mathrm{UJs}_{\vert K} = \left\{([\theta],[\omega]) \in K : (X,[\omega])\;\mathrm{is}\; \mathrm{uniformly}\;J^{\theta}-\mathrm{stable} \right\} = \left\{([\theta],[\omega]) \in K : \int_C \tau_{\theta,\omega} > 0 \right\}.
$$
(Where for the last equality we have used \cite{DatarPingali, Song}.
Combined with Theorem \ref{Thm main Jequation} this  entails that there exists a collection of curves $\left\{C_i\right\}$ of cardinality $\ell \leq k\rho(X)$ such that 
$$
\left(\mathrm{Jss} \setminus \mathrm{UJs}\right)_{\vert K} = \prod_{i = 1}^{\ell} \int_{C_i} \tau_{\theta,\omega} = 0,
$$
defining a real algebraic set in $K$. This is what we wanted to prove. 
\end{proof}

\begin{remark}
The conditions $\int_{C_i} \tau_{\theta,\omega} = 0$ define codimension $1$ subsets by a similar argument to Proposition \ref{Cor Z crit codim 1}. 
\end{remark}

\noindent To summarize the relationship between the different positivity/stability notions involved,  
optimally destabilizing curves exist for $J^{\theta}$-positivity (i.e. the condition that $\int_E \tau_{\theta,\omega} > 0$ for all curves $E \subset X$, equivalent to uniform J-stability, by \cite{GaoChen,DatarPingali,Song}), but on the algebro-geometric side optimally destabilizing (normal and relatively K\"ahler) test configurations do \emph{not} exist for $J^{\theta}$-stability. Moreover, $J^{\theta}$-stability turns out to be \emph{equivalent} to $J^{\theta}$-semistability over rational classes in $H^{1,1}(X,\mathbb{R}$, and the set $\mathrm{Jss}$ represents the closure of the (open, see \cite{SD6}) set $\mathrm{UJs}$ for uniform J-stability. The boundary $\partial \mathrm{UJs}$ is real algebraic and locally cut out by finitely many equations. A very interesting future direction is to similarly clarify the analogous relations for K-semistability/K-stability/uniform K-stability (see e.g. \cite{Hattori} for first results in this direction). 

\section{Flow aspects} \label{Section flow}

\noindent In this section we consider certain well-known theorems regarding geometric flows associated with the J-equation and the deformed Hermitian Yang-Mills equation, and remark on a possible relationship between the singularities that develop in each of these cases. 

\subsection{Remarks on a J-flow result of Song-Weinkove}
In \cite{SongWeinkove} the authors study a parabolic flow associated to the J-equation. Given a pair of Kähler classes $\beta,\alpha$ and Kähler forms $\theta\in\beta$ and $\omega_0\in \alpha$ on a compact Kähler manifold $X$ of complex dimension $n$, the authors study the \emph{J-flow} of the pair $\theta, \omega_0$, namely 
\begin{align}
    \frac{\partial \varphi(\cdot,t)}{\partial t} &= C_{\beta,\alpha} - \frac{n\theta\wedge (\omega_0 + \i\ddbar{\varphi(\cdot,t))^{n-1}}}{(\omega+\i\ddbar{\varphi(\cdot,t)})^n}\\
    \varphi(x,0) &= 0 \quad \textrm{for all } x \in X,
\end{align}
for a smooth function $\varphi:X\times [0,\infty) \to \RR.$ Here, 
$$
C_{\beta,\alpha} := \frac{n\int_X \beta\cdot\alpha^{n-1}}{\int_X \alpha^n}.
$$
The authors prove that this is a well-defined flow in 
$$
\mathcal H(\omega_0) = \left\{\varphi \in C^\infty(X,\RR)\; | \; \omega_0 + \i \ddbar{\varphi} > 0 \right\}
$$
which converges to a solution $\varphi_\infty$ of the J-equation 
\begin{equation}\label{eqn J equation}
n\theta\wedge(\omega_0 + \i \ddbar{\varphi_\infty})^{n-1}=C_{\beta,\alpha} \theta^n
\end{equation}
if a solution exists. If no solutions exist, then the authors also prove that the flow $\varphi(\cdot,t)$ becomes singular along certain divisors. More precisely, in the special case whereby $X$ is a surface, the authors prove the following result \cite[Theorem 1.4]{SongWeinkove}. 
\begin{theorem}[Song-Weinkove, \cite{SongWeinkove} Theorem 1.4] \label{Thm SongWeinkove result J flow} Let $X$ be a compact Kähler surface and let $\alpha, \beta$ be Kähler classes on $X$. Suppose the J-equation \eqref{eqn J equation}
does not admit a solution in $\alpha$. Then, there exists an effective $\RR$-divisor 
$$
D= \sum_{i=1}^\ell a_i E_i
$$
on $X$ such that the class $\tau_\mathrm J (\beta,\alpha)-D$ is Kähler, where 
$$
\tau_\mathrm J (\beta,\alpha) := C_{\beta,\alpha}\alpha - \beta.
$$ 
The irreducible components $E_i$ of $D$ are all curves of negative self-intersection.
Moreover, the  J-flow $\varphi(\cdot,t)$ remains bounded (uniformly in $t$) in the set $X\setminus S(D)$ where 
$$
S(D):= \bigcup_{i=1}^\ell E_i.
$$ 
Let $\tilde S$ denote the set given by the intersection of all the sets $S(D)$ as $D$ varies over the linear system $|D|$.  Then, there exists a sequence $(x_j,t_j)$ of points and times in $X\times [0,\infty)$ such that, as $j \to\infty$, $t_j\to\infty$ and the points $x_j$ get arbitrarily close to $\tilde S$, and 
$$
(|\varphi| + |\triangle_\theta \varphi|)(x_j,t_j) \to \infty.
$$
\end{theorem}

In fact, using the framework of the Zariski decomposition, we can make the above theorem a bit more precise as follows.

\begin{proposition}\label{prop Song-Weinkove improvement}
    In the above Theorem \ref{Thm SongWeinkove result J flow}, we can always take $D$ to be 
    $$
    D = N(\tau_\mathrm J (\beta,\alpha) - \varepsilon\beta),
    $$
    the negative part of the Zariski decomposition of $\tau_\mathrm J(\beta,\alpha) - \varepsilon\beta$ for any $\varepsilon > 0$ small enough. In particular, we may always take $\ell \leq \rho(X)$ and $\tilde S = S(D)$. If $X$ is moreover projective, then $\ell \leq \rho(X) - 1$. 
\end{proposition}
\begin{proof}
    The proof of Theorem \ref{Thm SongWeinkove result J flow} (Theorem 1.4 in \cite{SongWeinkove}) shows that in fact, the conclusion of the theorem holds for \emph{any} effective $\RR$-divisor such that the class $\tau_\mathrm J(\beta,\alpha) - D$ is a Kähler class. However, recall that the class $\tau_\textrm{J}(\beta,\alpha) \in \mathcal P^+_X$ and so for every $\varepsilon>0$ small enough, $\tau_\mathrm J (\beta,\alpha) - \varepsilon\beta$ is also in $\mathcal P^+_X$, and clearly 
$$
\tau_\textrm J (\beta,\alpha) - N(\tau_\textrm J (\beta,\alpha)-\varepsilon\beta) = Z(\tau_\textrm J (\beta,\alpha)-\varepsilon\beta) + \varepsilon\beta
$$
is a Kähler class for every $\varepsilon >0$ small enough. Writing 
$$
D = N(\tau_\textrm J (\beta,\alpha)-\varepsilon\beta) = \sum_{i=0}^\ell
a_i E_i
$$
for irreducible curves $E_i$ on $X$, we see that we can always take $\ell \leq \rho(X)$, because the intersection matrix of the negative part of the Zariski decomposition is always negative-definite. Moreover, according to \cite[Proposition 3.13]{Boucksomthesis}, the class represented by the negative part of the Zariski decomposition always admits precisely one positive current, so we see that $|D|$ admits the unique element $D$ in it. This means 
$$\tilde S = S(D) = \bigcup_{i=1}^\ell E_i.$$ If $X$ is projective, then $NS(X)_\RR$ contains a positive eigenvector of the intersection pairing, and we get $\ell \leq \rho(X) - 1$.
\end{proof}

\medskip

\subsection{Remarks on the line bundle mean curvature flow and a result of Takahashi} In the context of the dHYM equation, Jacob-Yau in \cite{JacobYau} have proposed the study of the so called \emph{line bundle mean curvature flow}. In the above notation, this flow can be written 
\begin{align}
    \frac{\partial \varphi(\cdot,t)}{\partial t} &= \arctan(\lambda_1(t))+\arctan(\lambda_2(t)) - \hat\Theta(\beta,\alpha)\\
    \varphi(x,0) &= 0 \quad \textrm{ for all } x\in X, 
\end{align}
for a smooth function $\varphi_0$. 
Here $\lambda_1(t), \lambda_2(t)$ are the eigenvalues of the endomorphism $A_t:TX^{1,0}\to TX^{1,0}$ determined by 
$$
\theta(A_tv,w) = (\omega_0 + \i\ddbar{\varphi(\cdot,t)})(v,w),
$$
and 
$$
\cot(\hat\Theta(\beta,\alpha))= \frac{\int_X (\beta^2-\alpha^2)}{\int_X 2\alpha\cdot\beta}. 
$$
Under suitable hypotheses, in \cite{TakahashiLMBCF}, Takahashi proves that the solution $\varphi(\cdot,t)$ of the flow remains bounded away from finitely many curves of negative self-intersection. More precisely, he proves the following theorem. (See also \cite[Theorem 1.4]{FuYauZhang} for a very similar result on the so-called \emph{dHYM flow} studied by Fu-Yau-Zhang.) 
\begin{theorem}[{Takahashi, \cite[Theorem 1.1]{TakahashiLMBCF}}] \label{Thm Takahashi LBMCF}
Suppose the initial data $\omega_0, \theta$, and $\hat\Theta = \hat\Theta(\beta,\alpha)$ satisfy
$$
\tilde\tau_\mathrm{dHYM}(\theta,\omega_0):= \omega_0 + \cot(\hat\Theta)\theta
$$
is a semi-positive form, and we have
\begin{equation}\label{eqn pointwise phase hypothesis}
\arctan(\lambda_1(0)) + \arctan(\lambda_2(0))>\frac{\pi}{2}
\end{equation}
and $\hat\Theta > \frac{\pi}{2}$, then there exist finitely many curves $\tilde E_1,\cdots,\tilde E_k$ of negative self-intersection such that the flow $\varphi(\cdot,t)$ converges to a function $\varphi_\infty$ which is smooth on the complement of the set
$$
S = \bigcup_{i=1}^k \tilde E_i.
$$
Moreover, the current defined by 
$$
F_\infty = \omega_0 + \i \ddbar{\varphi_\infty}
$$
satisfies the deformed Hermitian Yang-Mills equation in the sense of currents on all of $X$.
\end{theorem}

\noindent Here, we make the simple observation that the curves supplied by the above Theorem \ref{Thm Takahashi LBMCF} of Takahashi in fact comprise a subset of those supplied by Theorem \ref{Thm SongWeinkove result J flow} due to Song-Weinkove. More precisely, we have the following.

\begin{proposition} \label{Prop Takahashi improvement}
    In the setup of the above Theorem \ref{Thm Takahashi LBMCF}, we can take the set of curves $\left\{\tilde E_1, \cdots, \tilde E_k\right\}$ to be a subset of the set of curves $\left\{E_1,\cdots, E_\ell\right\}$ given by Theorem \ref{Thm SongWeinkove result J flow}. In particular, $k \leq \ell \leq \rho(X)$.
\end{proposition}
\begin{proof}
    Set 
    $$
    \tau_\mathrm{dHYM}(\beta,\alpha) := \alpha + \cot(\hat\Theta(\beta,\alpha))\beta.
    $$
    The proof of Theorem \ref{Thm Takahashi LBMCF} (see \cite[Theorem 1.1]{TakahashiLMBCF} or \cite[Theorem 1.4]{FuYauZhang}) shows that the curves $\tilde E_1,\cdots, \tilde E_k$ can be taken to be the irreducible components of any effective $\RR$-divisor $\tilde D = \sum_{i=1}^k \tilde a_i \tilde E_i$ such that
    $$
    \tau_\mathrm{dHYM}(\beta,\alpha) - \tilde D
    $$
    is a Kähler class. Just like in Proposition \ref{prop Song-Weinkove improvement}, we can take $\tilde D$ to be 
    $$
    \tilde D = N(\tau_\mathrm{dHYM}(\beta,\alpha) - \varepsilon\beta),
    $$
    the negative part of the Zariski decomposition of the class $\tau_\mathrm{dHYM}(\beta,\alpha) - \varepsilon\beta$ for any $\varepsilon >0 $ small enough. 
    But now observe that 
    $$
    \tau_\mathrm{dHYM}(\beta,\alpha) = \frac{\int_X\alpha^2}{2\int_X \beta\cdot\alpha} \tau_\mathrm J (\beta,\alpha) + \frac{\left(\int_X \alpha^2\right)\left(\int_X \beta^2\right)}{2 \int_X \beta\cdot\alpha}\beta 
    $$
    and so 
    \begin{align*}
    N(\tau_\mathrm{dHYM}(\beta,\alpha) - \varepsilon\beta) &= N\left(\frac{\int_X\alpha^2}{2\int_X \beta\cdot\alpha} \tau_\mathrm J (\beta,\alpha) - \varepsilon\beta + \frac{\left(\int_X \alpha^2\right)\left(\int_X \beta^2\right)}{2 \int_X \beta\cdot\alpha}\beta\right)\\
    &\leq \frac{\int_X\alpha^2}{2\int_X \beta\cdot\alpha} N\left(\tau_\mathrm J(\beta,\alpha)-\varepsilon C_{\beta,\alpha}\beta\right)+N\left(\frac{\left(\int_X \alpha^2\right)\left(\int_X \beta^2\right)}{2 \int_X \beta\cdot\alpha}\beta\right)\\
    &\leq \frac{\int_X\alpha^2}{2\int_X \beta\cdot\alpha} N\left(\tau_\mathrm J(\beta,\alpha)-\varepsilon C_{\beta,\alpha}\beta\right),
    \end{align*}
    by the convexity of the negative part of the Zariski decomposition (see \cite[Proposition 3.9 (i)]{Boucksomthesis}) and that fact that the negative part of a Kähler class is zero. The proof is now completed by taking $\varepsilon > 0$ small enough.
\end{proof}

\medskip

\subsection{Remarks on a dHYM flow result of Fu-Yau-Zhang}

An analogous result holds also in the case of the dHYM flow studied recently by Fu-Yau-Zhang \cite{FuYauZhang}, with the same proof:

\begin{proposition}
    In \cite[Theorem 1.4]{FuYauZhang}, the finitely many curves $E_i$ of negative self-intersection can be taken to be a subset of the set of curves $\left\{E_1,\cdots, E_\ell\right\}$ given by Theorem \ref{Thm SongWeinkove result J flow}. In particular the number of such curves is bounded above by $\rho(X)$.
\end{proposition}

\noindent Moreover, as previously remarked, if the initial data $(\theta,\omega_0)$ underlying the various flows discussed above is varied in such a way that $([\theta],[\omega_0])$ remains in a compact subset $K \subseteq \mathcal C_X \times \mathcal{C}_X$, and the hypotheses of the Theorems \ref{Thm SongWeinkove result J flow}, \ref{Thm Takahashi LBMCF} and \cite[Theorem 1.4]{FuYauZhang} are satisfied, the curves given by the theorems can be taken to lie in a \emph{uniform finite set} depending only on $K$.

The above observations, together with the well-known fact that the J-equation is the so-called `small volume limit' of the dHYM suggests the following interesting question: Suppose we fix Kähler classes $\beta$ and $\alpha$, and Kähler metrics $\theta\in \beta$ and $\omega_0 \in \alpha$, and study the flow, 
\begin{align}\label{eqn flow family}
    \frac{\partial \varphi(\cdot,t)}{\partial t} &= \frac{\arctan(r\lambda_1(t))+\arctan(r\lambda_2(t))}{r} - \hat\Theta(\beta,r\alpha)\\
    \varphi(x,0) &= 0 \quad \forall x\in X, 
\end{align}
for $r>0$. Then for $r>0$ very small, the dHYM equation 
$$
\operatorname{Im}\left(e^{-\i\hat\Theta(\beta,r\alpha)}(\theta+\i 
 r\omega)^2\right)=0
$$
is always solvable, because the class
$$
\tau_{\mathrm{dHYM}}(\beta,r\alpha) = r\alpha + \cot(\hat\Theta(\beta,r\alpha))\beta = r\alpha + \frac{\int_X(\beta^2-r^2\alpha^2)}{2r\int_X \beta\cdot\alpha} \beta
$$
is a Kähler class. As we increase $r$, at some finite radius $r_0 > 0$, the form $r\omega_0 + \cot(\hat\Theta(\beta,r\omega)\theta$ is semi-positive but not positive-definite, and we are in the setup of Theorem \ref{Thm Takahashi LBMCF} (provided the pointwise phase hypothesis \eqref{eqn pointwise phase hypothesis} is also satisfied). We then obtain finitely many curves $E_1,\cdots, E_{k_0}$ such that the flow converges to a function $\varphi_{r_0}$ which is smooth on the complement of their union. On the other hand, we also have 
$$
\frac{1}{r}\tau_\mathrm{dHYM}(\beta,r\alpha) = \frac{\int_X\alpha^2}{2\int_X \beta\cdot\alpha} \tau_\mathrm J (\beta,\alpha) + \frac{\left(\int_X \alpha^2\right)\left(\int_X \beta^2\right)}{2r^2 \int_X \beta\cdot\alpha}\beta 
$$
for every $r> 0$. So, on the opposite extreme, we can make $r$ large enough so that the class 
$$
\frac{1}{r}\left(\tau_\mathrm{dHYM}(\beta,r\alpha) - \frac{\left(\int_X \alpha^2\right)\left(\int_X \beta^2\right)}{r \int_X \beta\cdot\alpha}\beta\right) = \frac{\int_X\alpha^2}{2\int_X \beta\cdot\alpha} \tau_\mathrm J (\beta,\alpha) - \frac{\left(\int_X \alpha^2\right)\left(\int_X \beta^2\right)}{2r^2 \int_X \beta\cdot\alpha}\beta 
$$
is in $\mathcal P^+_X$. Then, the proof of Proposition \ref{prop Song-Weinkove improvement} shows us that the curves appearing in the negative part of this class can be taken to be the curves of Theorem \ref{Thm SongWeinkove result J flow}, along which the J-flow becomes singular. One can therefore ask how the singularities of the flow change as we vary $r\in [r_0,\infty)$. In particular, it might give more insight into the nature of the flows to consider the finitely many values $r_0 < r_1 < \cdots < r_s$ for which the flow \eqref{eqn flow family} develops \emph{new} singularities.


\medskip

\begin{thebibliography}{1}
\bibitem{Bayer} A. Bayer, Polynomial Bridgeland stability conditions and the large volume limit, Geom. Topol. 13 (2009), no. 4, 2389–2425.
\bibitem{BlumLiuZhou} H. Blum, Y. Liu, C. Zhou, Optimal destabilization of K-unstable Fano varieties via stability thresholds, to appear in Geom. Topo, arXiv:1907.05399 [math.AG].
\bibitem{Boucksomthesis} S. Boucksom, {C\^ones} positifs des {vari\'et\'es} complexes compactes, PhD thesis, Math\'ematiques [math]. Université Joseph-Fourier-Grenoble I, 2002. tel-00002268. 
\bibitem{BHJ1} S. Boucksom, T. Hisamoto, M. Jonsson, Uniform K-stability, Duistermaat-Heckman measures and singularities of pairs, Ann. Inst. Fourier (Grenoble)
{67} (2017) no. 2, 743-841.
\bibitem{Bridgeland} T. Bridgeland, Stability conditions on triangulated categories, Ann. of Math. (2) {166} (2007), no. 2, 317–345.
\bibitem{BridgelandK3} T. Bridgeland, Stability conditions on K3 surfaces, Duke Math. J. {141} (2008), no. 2, 241-291.
\bibitem{GaoChen} G. Chen, The J-equation and the supercritical deformed Hermitian-Yang-Mills equation, Invent. Math. {225} (2021), no. 2, 529–602.
\bibitem{Calabi} E. Calabi, The space of  {K\"ahler metrics}, Proceedings of the International Congress of Mathematics 1954 2 (1954), 206–
\bibitem{Chen2000} X.X. Chen, On the lower bound of the {Mabuchi} {K-energy} and its application, Int. Math. Res. Not. {12} (2000), 607-623.
\bibitem{ChuCollinsLee} J. Chu, T. Collins, M. Lee, The space of almost calibrated $(1,1)$ forms on a compact K\"ahler manifold, Geom. Topo. {25} (2021), no. 5, 2573-2619.
\bibitem{ChuLeeTakahashi} J. Chu, M. Lee, R. Takahashi, A Nakai-Moishezon type criterion for supercritical deformed Hermitian-Yang-Mills equation, to appear in J. Differential Geom. Preprint arxiv:2105.10725v3 [math.DG].
\bibitem{CollinsJacobYau} T. Collins, A. Jacob, S.-T. Yau, $(1,1)$ forms with specified Lagrangian phase: a priori estimates and algebraic obstructions, Camb. J. Math. {8} (2020), 407-452.
\bibitem{CollinsShi} T. Collins, Y. Shi, Stability and the deformed Hermitian-Yang-Mills equation, Preprint arXiv:2004.04831 [math.DG]. 
\bibitem{CollinsGabor} T. Collins, G. Sz\'ekelyhidi, Convergence of the {J}-flow on toric manifolds, J. Differential Geom. {107} (2017) no. 1, 47-81.
\bibitem{CollinsXieYau} T. Collins, D. Xie, S.-T. Yau, The deformed Hermitian-Yang-Mills equation in geometry and physics, Geometry and physics. Vol. I, 69-90, Oxford Univ. Press, Oxford, 2018.
\bibitem{DatarPingali} V. Datar, V. Pingali, A numerical criterion for generalised Monge-{Amp\`ere} equations on projective manifolds, Geom. Funct. Anal. {31} (2021), no. 4, 767–814.
\bibitem{DelcroixJubert} T. Delcroix, S. Jubert, An effective weighted K-stability condition for polytopes and semisimple principal toric fibratons, Preprint arXiv:2202.02996v2 [math.DG].
\bibitem{Delcroix} T. Delcroix, K-Stability of Fano spherical varieties, Ann. Sci. Éc. Norm. Supér. (4) {53} (2020), no. 3, 615–662.
\bibitem{DelcroixUKs} T. Delcroix, Uniform K-stability of polarized spherical varieties, Preprint arXiv:2009.06463v2 [math.AG]
\bibitem{DemaillyPaun} J.-P. Demailly and M. Păun, Numerical characterization of the K\"ahler cone of a compact K\"ahler manifold,
Ann. of Math. (2) {159} (3) (2004), 1247-1274.
\bibitem{DervanUKs}  R. Dervan, Uniform stability of twisted constant scalar curvature {K\"ahler} metrics, Int. Math. Res. Not. IMRN {2016}, no. 15, 4728–4783.
\bibitem{Dervan} R. Dervan, Stability conditions in geometric invariant theory, Preprint arXiv:2207.04766v1 [math.AG].
\bibitem{Dervanoptimal} R. Dervan, K-semistability of optimal degenerations, Q. J. Math. {71} (2020), no. 3, 989–995.
\bibitem{DMS} R. Dervan, J.B. McCarthy, L.M. Sektnan, $Z$-critical connections and Bridgeland stability conditions, Preprint arXiv:2012.10426v3 [math.DG]. 
\bibitem{DervanRoss} R. Dervan, J. Ross, {K-stability} for K\"ahler manifolds, Math. Res. Lett. {24} (2017), no. 3, 689–739. 
\bibitem{DonaldsonJobservation} S. K. Donaldson, Moment maps and diffeomorphisms, Asian J. Math. {3} (1999), no. 1, 1–15.
\bibitem{Donaldsontoric} S. K. Donaldson, Scalar curvature and stability of toric varieties, J. Differential Geom. {62} (2002), no. 2, 289–349.
\bibitem{DonaldsonCalabi} S. K. Donaldson, Lower bounds on the Calabi functional, J. Differential Geom., {70} (2005), no 3, 453–472.
\bibitem{FuYauZhang} J. Fu, S.-T. Yau, D. Zhang, A deformed Hermitian Yang-Mills flow, Preprint arXiv:2105.13576v3 [math.DG].
\bibitem{Futaki} A. Futaki, An obstruction to the existence of {Einstein-K\"ahler} metrics, Invent. Math. {73} (1983), no. 3, 437–443.
\bibitem{Hattori} M. Hattori, A decomposition formula for J-stability and its applications, Preprint arXiv:2103.04603v2 [math.AG]. 
\bibitem{Hisamotooptimal} T. Hisamoto, Geometric flow, multiplier ideal sheaves and optimal destabilizer for Fano manifold, Preprint arXiv:1901.08480v3 [math.DG]. 
\bibitem{JacobYau} A. Jacob, S.-T. Yau, A special Lagrangian type equation for holomorphic line bundles, Math. Ann. {369} (2017), no. 1-2, 869–898.
\bibitem{Lamari} A. Lamari, Le {c\^one} {k\"ahl\'erien} d'une surface, J. Math. Pures Appl. (9) {78} (1999), no. 3, 249–263.
\bibitem{LejmiGabor} M. Lejmi, G. Sz\'ekelyhidi, The {J}-flow and stability, Adv. Math. {274} (2015), 404–431.
\bibitem{Mabuchi} T. Mabuchi, Some symplectic geometry on compact K\"ahler manifolds I, Osaka J. Math. {24} (1987), no. 2, 227–252.
\bibitem{McCarthy} J. B. McCarthy, Stability conditions and canonical metrics, PhD thesis (2022). 
\bibitem{Pingali} P. Pingali, The deformed Hermitian Yang-Mills equation on three-folds, Anal. PDE {15} (2022), no. 4, 921–935.
\bibitem{RossThomas} J. Ross, R. Thomas, An obstruction to the existence of constant scalar curvature {K\"ahler} metrics, J. Differential Geom. {72} (2006), no. 3, 429–466.
\bibitem{SektnanTipler} L.M. Sektnan, C. Tipler, Analytic K-semistability and wall-crossing, Preprint arXiv:2212.08383 [math.DG].
\bibitem{Gabortoric} G. {Sz\'ekelyhidi}, Optimal test-configurations for toric varieties, J. Differential Geom. {80} (2008), no. 3, 501–523.
\bibitem{SD1} Z. Sj\"ostr\"om Dyrefelt, K-semistability of cscK manifolds with transcendental cohomology class, J. Geom. Anal. {28} (2018), no. 4, 2927–2960. 
\bibitem{SD4} Z. Sj\"ostr\"om Dyrefelt, Optimal lower bounds for Donaldson's J-functional, Adv. Math. {374} (2020), 107271, 37 pp.
\bibitem{SD6} Z. Sj\"ostr\"om Dyrefelt, Openness of uniform K-stability in the K\"ahler cone, Preprint arXiv:2011,14806v1 [math.DG]. 
\bibitem{Zakthesis} Z. Sj\"ostr\"om Dyrefelt, K-stabilit\'e et vari\'et\'es k\"ahleriennes avec classe transcendante, PhD thesis, Universit\'e de Toulouse, 2017. http://thesesups.ups-tlse.fr/3577/.
\bibitem{SongWeinkove} J. Song, B. Weinkove, On the convergence and singularities of the {J}-flow with applications
to the {Mabuchi} energy, Comm. Pure Appl. Math. {61} (2008), no. 2, 210–229.
\bibitem{Song} J. Song, Nakai-Moishezon criterions for complex Hessian equations, Preprint arXiv:2012.07956v1 [math.DG]. 
\bibitem{Jacopo} J. Stoppa,
Twisted constant scalar curvature {K\"ahler} metrics and {K\"ahler} slope stability, J. Differential Geom. 83 (3): 663-691 (November 2009). DOI: 10.4310/jdg/1264601038.  
\bibitem{TakahashiLMBCF} R. Takahashi, Collapsing of the line bundle mean curvature flow on {K\"ahler} surfaces,
Calc. Var. Partial Differential Equations {60} (2021), no. 1, paper no. 27, 18 pp.
\bibitem{Takahashioptimal} R. Takahashi, The Kähler-Ricci flow and quantitative bounds for
Donaldson-Futaki invariants of optimal degenerations, Proc. Amer. Math. Soc. {148} (2020), no. 8, 3527–3536.
\bibitem{Tian} G. Tian, {K\"ahler}-{Einstein} metrics with positive scalar curvature, Invent. Math. {130} (1997), no. 1, 1–37.
\bibitem{Yau} S.-T. Yau, On the Ricci curvature of a compact K\"ahler manifold and the complex Monge-Amp\`ere equation I, Comm. Pure Appl. Math. {31} (1978), no. 3, 339–411.
\bibitem{Zariski} O. Zariski, The theorem of Riemann-Roch for high multiples of an effective divisor on an algebraic surface, Ann. of Math. (2) {76} (1962), 560–615.
\bibitem{ZZhang} Z. Zhang, Optimal destabilizing centers and equivariant K-stability, Invent. Math. {226} (2021), no. 1, 195–223.
\bibitem{CZhou2} C. Zhou, On the shape of K-semistable domain and wall crossing for K-stability, Preprint arXiv:2302.13503 [math.AG].
\bibitem {CZhou} C. Zhou, Chamber decomposition for K-semistable domains and VGIT, Preprint arXiv:2303.10963 [math.AG].
\end{thebibliography}
\end{document}